  \documentclass[12pt]{amsart}
\usepackage{kotex}
\usepackage{amsmath}
\usepackage{amsfonts}
\usepackage{amscd}

\usepackage{amssymb}
\usepackage{graphicx}
\usepackage{caption}
\usepackage{subcaption}
\usepackage{dsfont}
\usepackage{color}
\usepackage{hyperref}
\usepackage{bbm}
\usepackage{a4wide}
\usepackage{sseq}
\usepackage{tikz-cd}
\usepackage[abs]{overpic}

\usepackage{psfrag}
\usepackage{marvosym}
\usepackage{amssymb}
\usepackage{amsthm}
\usepackage{mathrsfs}
\usepackage{enumitem}
\usepackage{bbm}
 \usepackage{comment}
\usepackage{ytableau}
\usepackage{hyperref}

\usepackage{calligra}
\usepackage{mathrsfs}
\DeclareMathAlphabet{\mathcalligra}{T1}{calligra}{m}{n}
\DeclareFontShape{T1}{calligra}{m}{n}{<->s*[1.5]callig15}{}

\newtheorem{theorem}{Theorem}[section]

\newtheorem{lemma}[theorem]{Lemma}

\theoremstyle{definition}
\newtheorem{definition}[theorem]{Definition}

\newtheorem{remark}[theorem]{Remark}

\newtheorem{theorem-definition}[theorem]{Theorem-Definition}

\numberwithin{equation}{section}

\newcommand{\CC} {\mathbb{C}}

\newcommand{\NN} {\mathbb{N}}

\newcommand{\RR} {\mathbb{R}}

\newcommand{\ZZ} {\mathbb{Z}}

\newcommand {\shD} {\mathcal{D}}

\newcommand {\shH} {\mathcal{H}}
\newcommand {\shI} {\mathcal{I}}

\newcommand {\shL} {\mathcal{L}}

\newcommand {\shP} {\mathcal{P}}

\newcommand {\fot}  {\mathfrak{t}}

\newcommand {\foV} {\mathfrak{V}}

\newcommand{\sExt}{\mathscr{E} \kern -1pt xt}

\newcommand {\GL} {\operatorname{GL}}

\newcommand {\sHom}{\mathscr{H}\kern-5pt\mathcalligra{om}}

\renewcommand {\Im} {\operatorname{Im}}

\renewcommand {\ker } {\operatorname{Ker}}

\newcommand{\Diff}{\operatorname{Diff}}

\newcommand {\rank} {\operatorname{rank}}

\newcommand{\sTor}{\mathscr{T} \kern -3pt or}

\begin{document}

\title[ Geodesic rays in the space of K\"ahler metrics with T-symmetry]{Geodesic rays in the space of \\
K\"ahler metrics with T-symmetry}	

\author[Leung]{Naichung Conan Leung,}
\address{The Institute of Mathematical Sciences and Department of Mathematics\\ The Chinese University of Hong Kong\\ Shatin\\ Hong Kong}
\email{leung@math.cuhk.edu.hk}
	
\author[Wang]{Dan Wang}
\address{School of Mathematical Sciences\\ East China Normal University \\Shanghai \\China}
\address{The Institute of Mathematical Sciences and Department of Mathematics\\ The Chinese University of Hong Kong\\ Shatin\\ Hong Kong}
\email{dwang116@link.cuhk.edu.hk}
\thanks{}
\maketitle

\begin{abstract}
Let $(M, \omega, J)$ be a K\"ahler manifold, equipped with an effective
 Hamiltonian torus action $\rho: T \rightarrow \Diff(M, \omega, J)$ by isometries with moment map $\mu: M \rightarrow \mathfrak{t}^{*}$. We first construct a singular mixed polarization $\mathcal{P}_{\mathrm{mix}}$ on $M$. Second we construct a one-parameter family of complex structures $J_{t}$ on $M$ which are compatible with $\omega$. Furthermore the path of corresponding K\"ahler metrics $g_{t}$ is a complete geodesic ray in the space of K\"ahler metrics of $M$, when $M$ is compact.
Finally, we show that the corresponding family of K\"ahler polarizations $\mathcal{P}_{t}$ associated to $J_{t}$ converges to $\mathcal{P}_{\mathrm{mix}}$ as $t \rightarrow \infty$. \end{abstract}



\section{Introduction}
An important problem in geometric quantization is to understand the relationship among quantizations associated to different polarizations on a compact symplectic manifold $(M,\omega)$. If $M$ is K\"ahler, its complex structure $J$ would define a K\"ahler polarization $\shP_{J}$ on it. In this paper, we assume $M$ also admits a Hamiltonian torus action by isometries. We first construct a (possibly singular) polarization $\shP_{\mathrm{mix}}$ on $M$ by combining the Hamiltonian action with the K\"ahler polarization $\shP_{J}$. Second we construct a one-parameter family of complex structures $J_{t}$ on $M$ which are compatible with $\omega$. Furthermore the path of corresponding K\"ahler metrics $g_{t}$ is a complete geodesic ray in the space of K\"ahler metrics of $M$, when $M$ is compact.
Finally, we show that the corresponding family of K\"ahler polarizations $\mathcal{P}_{t}$ associated to $J_{t}$ converges to $\mathcal{P}_{\mathrm{mix}}$ as $t \rightarrow \infty$.

 When $M$ is toric, this is a result of Baier, Florentino, Mour\~{a}o, and Nunes in \cite{BFMN}. 
In this case, $\shP_{\mathrm{mix}}$ is essentially the real polarization defined by moment map. Unlike the toric cases, we can not use the explicit K\"ahler structures given by symplectic potentials (see \cite{ Ab1, Ab2,Gui1,Gui2}). Instead we will use the imaginary time flow approach introduced by Mour\~{a}o and Nunes in \cite{MN}, the relationship between K\"ahler polarizations and moment maps studied by Burns and Guillemin in \cite{BG}, and holomorphic slices investigated by Sjamaar in \cite{Sj}.

Throughout this paper, we assume the following:

 \begin{enumerate}
\item[$(*_{0}):$]  $(M, \omega, J)$ is a K\"ahler manifold of real dimension $2m$, equipped with an effective Hamiltonian $n$-dimensional torus action $\rho: T^{n} \rightarrow \Diff(M, \omega, J)$ by isometries with moment map $\mu: M \rightarrow \fot^{*}$.
\end{enumerate} 
$$
\begin{tikzcd}[row sep=3em, column sep=3em]
T^{n} \arrow[bend left]{r}{\rho} & (M, \omega) \arrow{r}{\mu} & \fot^{*}.
\end{tikzcd}
$$

From the Hamiltonian action, we have two singular distributions $\shD_{\CC}=(\ker d\mu) \otimes \CC$ and $\shI_{\CC}= (\Im d\rho)\otimes \CC$ on $M$ which are smooth on the open dense subset $\check{M}$ consisting of $n$-dimensional orbits in $M$. That is
$$\check{M}= \{p \in M| \dim H_{p} =0\},$$
where $H_{p}$ is the stabilizer of $T^{n}$ at point $p \in M$. 
Combining with the K\"ahler polarization $\shP_{J}$, we define the following singular distribution
$$\shP_{\mathrm{mix}}=(\shP_{J} \cap \shD_{\CC}) \oplus \shI_{\CC}$$ (see Definition \ref{def3-0-1}) on $M$.

 \begin{theorem}(Theorem \ref{thm3-0-5})
Under the assumption $(*_{0})$, we have $\shP_{\mathrm{mix}}$ is a singular polarization and $\mathrm{rk}_{\RR}(\shP|_{\check{M}})=n$.
\end{theorem}
 In order to construct a one-parameter family of K\"ahler polarizations joining $\shP_{J}$ and $\shP_{\mathrm{mix}}$, we need to choose a convex function on $\fot^{*}$.
 
   \begin{enumerate}
\item[$(*):$] Assume $(*_{0})$ and pick a strictly convex function $\varphi: \fot^{*} \rightarrow \RR$. We denote the Hamiltonian vector field associated to the composition $\varphi \circ \mu$ by $X_{\varphi}$. 
\end{enumerate} 
Then the imagine time flow $e^{-itX_{\varphi}}$ would give us a family of complex structures $J_{t}$ on $M$ for small $t$ (see \cite{MN}). Our next result show that $J_{t}$ exists for all $t \ge 0$.
 
 \begin{theorem}(Theorem \ref{thm3-3})
Under the assumption $(*)$, for any $t >0$, there exists a complex structure $J_{t}$ given by applying $e^{-itX_{\varphi}}$ to $J$-holomorphic coordinates and a unique biholomorphism:
 $$\psi_{t}: (M,J_{t}) \rightarrow (M,J).$$
\end{theorem}

 First we prove long time existence for the imaginary time flow $e^{-itX_{\varphi}}$ on local models. Such local models are given by the work of Sjamaar on holomorphic slices in \cite{Sj}. In this proof, we use the work of Burns and Guillemin \cite{BG} on the relationship between K\"ahler potentials and moment maps in our setting. Second we need to prove a commuting formula (see equation \ref{eq-com}) which guarantees that these local models can be glued together and gives a complex structure $J_{t}$ on the whole manifold $M$. 
 
 In Theorem \ref{thm3-4-1}, we verify $(M,\omega, J_{t})$ is a K\"ahler manifold. This gives us a one-parameter family of K\"ahler metrics $g_{t}=\omega(-,J_{t})$ on $(M, \omega)$. As observed by Donaldson in \cite{Do1}, these can be regarded as a family of K\"ahler metrics within a fixed K\"ahler class $[\omega]$ on a fixed complex manifold $(M, J)$. The space $\shH_{\omega}$ of these K\"ahler metrics was studied by Semmes in \cite{Se} and Donaldson in \cite{Do1}. In particular, $\shH_{\omega}$ is an infinite dimensional symmetric space of nonpositive curvature. The family of K\"ahler metrics $g_{t}$'s we constructed above is always an geodesic ray in $\shH_{\omega}$ by using the work of Mour\~{a}o and Nunes in \cite{MN}.
 
 \begin{theorem}(Theorem \ref{thm3-4-1})
Under the assumption $(*)$, for any $t \ge 0$, $(M,\omega, J_{t})$ is a K\"ahler manifold. Moreover the path of K\"ahler metrics $g_{t}=\omega(-,J_{t}-)$ is a complete geodesic ray in the space of K\"ahler metrics of $M$. \end{theorem}

As $t$ goes to infinity, even through the metrics $g_{t}$'s does not have a limit, the corresponding polarizations $\shP_{t}$ does converge to a mixed polarization, namely $\shP_{\mathrm{mix}}$.
Thus we obtain a complete geodesic ray joining $\shP_{J}$ and $\shP_{\mathrm{mix}}$.
 
\begin{theorem}(Theorem \ref{thm3-5})
Under the assumption $(*)$,
let $J_{t}$ be the one-parameter family of complex structures constructed in Theorem \ref{thm3-3}. Then we have
$$\lim_{t\rightarrow \infty} \shP_{t}= \shP_{\mathrm{mix}}.$$
That is, $\lim_{t\rightarrow \infty} (\shP_{t})_{p}= (\shP_{\mathrm{mix}})_{p}$,  where the limit is taken in the Lagrangian Grassmannian of the complexified tangent space at point $p \in M$. 
\end{theorem}
We first reduce the proof to local models by using Sjamaar's work on the existence of holomorphic slices. Then we use Burns and Guillemin's results, plurisubharmonicity of $T^{n}$-invariant K\"ahler potential, and the convexity of $\varphi$ to show the existence of $\lim_{t\rightarrow \infty} \shP_{t}$.

There are many previous works by others on closely related problems for toric varieties \cite{BFMN, CS, KMN1,KMN4}, flag varieties \cite{GS3, HK}, cotangent bundles of compact Lie groups \cite{FMMN1, FMMN2, Hal, KMN2, MNP}, toric degenerations \cite{HHK, HK3}, and so on \cite{An1,An2,An3,BBLU, CLL, HK1, HK2, Hi, KW, LY1, LY2, MNR, RZ1, Th}.
 In the sequels, we will study geometric quantizations of these mixed polarizations in \cite{LW2} and analyze limit of geometric quantizations of K\"ahler polarizations along the complete geodesic rays studied in this paper in the case of the K\"ahler manifolds with $T$-symmetry in \cite{LW3}.
\subsection*{Acknowledgement}
We would like to thank Siye Wu for insightful comments and helpful discussions. We also thank the referees for valuable comments and suggestions for improvement.
This research was substantially supported by grants from the Research Grants Council of the Hong Kong Special Administrative Region, China (Project No. CUHK14301619 and CUHK14301721) and a direct grant from the Chinese University of Hong Kong.
\section{Preliminaries}
 
 In this section, we review results needed in the proof of Theorem \ref{thm3-3} including Burns-Guillemin's theorem in \cite{BG} and Sjamaar's holomorphic slices in \cite{Sj}.
 \subsection{Hamiltonian action}
Let $(M,\omega )$ be a symplectic manifold. For $f\in C^{\infty }(M,\mathbb{R
})$, the Hamiltonian vector field $X_{f}$ associated to $f$ is determined by 
$\imath _{X_{f}}\omega =-df$. Then the Poisson bracket of two functions $f, g \in C^{\infty}(M;\RR)$ can be defined by:
 $\{f, g\} = \omega(X_{f}, X_{g})$ and $$\psi:  (C^{\infty}(M),\{\}) \rightarrow \mathrm{Vect}(M,\omega)$$ is a Lie algebra homomorphism. Let $T^{n}$ be a torus of real dimension $n$
and $\rho :T^{n}\rightarrow \mathrm{Diff}(M,\omega )$ an action of $T^{n}$
on $M$ which preserves $\omega $. Differentiating $\rho $ at the identity
element, we have 
\begin{equation*}
d\rho :\mathfrak{t}\rightarrow \mathrm{Vect}(M,\omega ),~~~~\xi \mapsto \xi
^{\#}
\end{equation*}
where  $\mathfrak{t}$ is the Lie algebra of $T^{n}$ and $\xi ^{\#}$ is
called the fundamental vector field associated to $\xi $. 
The action of $T^{n}$ on $M$ is said to be {\em Hamiltonian} if $d\rho$ factors through $\psi$.
This gives a $T^{n}$-equivariant map
$\mu :M\rightarrow \mathfrak{t}^{* }$
called the \emph{moment mapping}, satisfying:
$$\omega(-, \xi^{\#})=d\mu^{\xi}.$$
\subsection{Polarizations on symplectic manifolds}

A step in the process of geometric quantization is to choose a polarization.
We first recall the definitions of distributions and polarizations on
symplectic manifolds $(M,\omega )$ (See \cite{Wo}). All polarizations
discussed in this subsection are smooth.

\begin{definition}
\label{def4-1}A \emph{complex distribution} $\mathcal{P}$ on a manifold $M$
is a complex sub-bundle of the complexified tangent bundle $TM\otimes 
\mathbb{C}$. 
When $\left( M,\omega \right) $ is a symplectic manifold, such a $\mathcal{P}
$ is a \label{def4-2}  \emph{complex polarization} if it satisfies the
following conditions:
\begin{enumerate}
\item $\mathcal{P}$ is involutive, i.e. if $u,v \in \Gamma(M,\mathcal{P})$,
then $[u,v] \in \Gamma(M,\mathcal{P})$;
\item for every $x \in M$, $\mathcal{P}_{x} \subseteq T_{x}M \otimes {
\mathbb{C}}$ is Lagrangian; and
\item  $\mathrm{rk}_{\RR}\left( \mathcal{P}\right):=\mathrm{rank}(\mathcal{P}\cap \overline{
\mathcal{P}}\cap TM)$ is constant.
\end{enumerate}
Furthermore, $\shP$ is called
\begin{enumerate}
\item  [$\cdot$]{\em real polarization}, if $\mathcal{P}=\overline{\mathcal{P}}$,
i.e. $\mathrm{rk}_{\RR}\left( \mathcal{P}\right) =m$; 
\item  [$\cdot$] {\em K\"ahler polarization}, if $\mathcal{P}\cap \overline{
\mathcal{P}}=0$, i.e. $\mathrm{rk}_{\RR}\left( \mathcal{P}\right) =0$; 
\item [$\cdot$] {\em mixed polarization}, if $0 < \rank(\shP \cap \overline{\shP} \cap TM) < m$, i.e. $0<\mathrm{rk}_{\RR}\left( \mathcal{P}\right) <m$. 
\end{enumerate}
\end{definition}
\subsection{Burns-Guillemin's theorem}

In this subsection, we recall an equivariant Darboux theorem for K\"ahler forms on $N = U \times T_{\mathbb{C}}^{n}$ and the relationship between Kähler potentials and moment maps proved by Burns and Guillemin in \cite{BG}, as follows, where $U$ is an open and convex subset of $\mathbb{C}^{m-n}$. Assume $T^{n}$ act on $N$ is given by the standard multiplication of $T^{n}$ on $T^{n}_{\CC}$. Let $w_{1}, \cdots ,w_{m-n}$ and $z_{1}, \cdots, z_{n}$ be the standard coordinate functions on $U$ and $T^{n}_{\CC}$. Let $\omega$ be any $T^{n}$-invariant K\"ahler form, which is Hamiltonian with respect to the action of $T^{n}$.  
Burns and Guillemin in \cite{BG} showed that:

\begin{enumerate}
\item $\omega= \sqrt{-1} \partial \bar{\partial} \rho$,
 where $\rho$ is a $T^{n}$-invariant function. Namely,
\begin{equation}\label{eq-po} \rho = \rho(w_{1},\cdots,w_{m-n},t_{1},\cdots,t_{n}), t_{i}=|z_{i}|^{2}. \end{equation}
\item If $\rho_{1}$ and $\rho_{2}$ are two $T^{n}$-invariant functions such that $\omega= \sqrt{-1} \partial \bar{\partial} \rho_{i}$, then there exist $\lambda_{i} \in \RR$ and a holomorphic function $Q$ on $U$ such that 
$$\rho_{2} -\rho_{1} = \sum_{i} \lambda_{i} \log t_{i} + \mathrm{Re} Q.$$
\end{enumerate}

\begin{theorem}\label{thm2-0} \cite[Theorem 3.1] {BG} Let $\rho$ and $\rho_{1}$ be strictly plurisubharmonic functions of the form (\ref{eq-po}). If the symplectic forms and moment maps associated with $\rho$ and $\rho_{1}$ are the same, then there exists a holomorphic function $Q$ on $U$ such that 
$$\rho_{1}= \rho +\mathrm{Re}Q.$$
Let $\mu=(\mu_{1}, \cdots, \mu_{n})$ be the moment map associated with the action of $T^{n}$ on $N$. Under the change of variables $t_{j}=e^{s_{j}}$, where $t_{j}=|z_{j}|^{2}$, one has:
$$\mu_{j} = \frac{\partial}{\partial s_{j}} f(w_{1},\cdots,w_{m-n}, s_{1}, \cdots,s_{n} )=  t_{j} \frac{\partial}{\partial t_{j}} \rho(w_{1},\cdots,w_{m-n}, t_{1},\cdots,t_{n}).$$
where $f(w_{1},\cdots,w_{m-n}, s_{1}, \cdots,s_{n} )= \rho(w_{1},\cdots,w_{m-n}, t_{1},\cdots,t_{n})$.
\end{theorem}

\subsection {Sjamaar's holomorphic slices} In this subsection, we recall Sjamaar's work (see \cite{Sj}) on the
existence of holomorphic slices. Let $G$ be a compact, connected Lie group.

\begin{theorem}
\label{thm2-2}\cite[theorem 1.12]{Sj} Let $M$ be a K\"{a}hler
manifold with a Hamiltonian $G$-action by holomorphic isometries. Given any
point $p$ lying on an isotropic $G$-orbit in $M$, there exists a slice at $p$
for the $G^{\mathbb{C}}$-action. 
\end{theorem}

We recall the definition of slices.

\begin{definition}
A slice at $p \in M$ for the $G^{\mathbb{C}}$-action is a locally closed
analytic subspace $S$ of $M$ with the following properties:

\begin{enumerate}
\item $p \in S$;

\item the saturation $G^{\mathbb{C}}S$ of $S$ is open in $M$;

\item $S$ is invariant under the action of the stabilizer $(H^{\mathbb{C}
})_{p}$;

\item the natural $G^{\mathbb{C}}$-equivariant map from $G^{\mathbb{C}}\times_{(H^{\mathbb{C}})_{p}} S$ into $M$, which sends $[g, y]$ to the point $g\cdot y$, is an analytic isomorphism onto 
$G^{\mathbb{C}}S$.
\end{enumerate}
\end{definition}

\begin{remark}
The condition that $p$ lying on an isotropic $G$-orbit in $M$ is equivalent
to $\mu \left( p\right) \in \mathfrak{g}^{\ast }$ is fixed under the
coadjoint action of $G$ \cite[theorem 1.12]{Sj}. Therefore every orbit is
isotropic if $G=T^{n}$ is Abelian, as in our situation. 
\end{remark}

\section{Main results}
\subsection{Construction of polarizations $\shP_{\mathrm{mix}}$ on K\"ahler manifolds with $T$-symmetry} \label{sec1}

 In this paper, we assume $(*_{0})$, i.e. $(M, \omega, J)$ is a compact K\"ahler manifold of real dimension $2m$, equipped with an effective Hamiltonian $n$-dimensional torus action $\rho: T^{n} \rightarrow \Diff(M, \omega, J)$ by isometries with moment map $\mu: M \rightarrow \fot^{*}$.
$$
\begin{tikzcd}[row sep=3em, column sep=3em]
T^{n} \arrow[bend left]{r}{\rho} & (M, \omega) \arrow{r}{\mu} & \fot^{*}.
\end{tikzcd}
$$
Let $H_{p}$ be the stabilizer of $T^{n}$ at point $p \in M$. Denote by $\check{M}$ the disjoint union of $n$-dimensional orbits, that is,
$$\check{M}= \{p \in M| \dim H_{p} =0\},$$
which is an open dense subset in $M$.
 To construct a geometric quantization, a crucial step is to choose a polarization (see Definition \ref{def4-2}) $\shP \subset TM\otimes \CC$, which is an integrable Lagrangian subbundle of the complexified tangent bundle. There is a natural K\"ahler polarization $\shP_{J}=T^{0,1}_{J}$ induced by the complex structure $J$.
  When $M$ a toric variety, i.e. $n=m$, it admits a singular real polarization given by $\mu$. 
 We are going to construct a singular polarization $\shP_{\mathrm{mix}}$ on $M$ by combining the $T^{n}$-action with the K\"ahler polarization $\shP_{J}$. We first construct a singular distribution of the form $$\shP_{\mathrm{mix}}=(\shP_{J} \cap \shD_{\CC}) \oplus \shI_{\CC}$$ (see Definition \ref{def3-0-1}). Then (see Theorem \ref{thm3-0-5}) we show that $\shP_{\mathrm{mix}}$ is a singular  polarization and smooth on $\check{M}$ with $\dim(\shP_{\mathrm{mix}} \cap \bar{\shP}_{\mathrm{mix}}  \cap TM)_{p}=n$ for any $p \in \check{M}$. 

We first define what we meant by a singular distribution (and polarizations).

\begin{definition}\label{def4-3}
 $\shP \subset  TM \otimes\CC$ is a {\em singular complex distribution} on $M$ if it satisfies:
$\shP_{p} $ is a vector subspace of $ T_{p}M \otimes \CC$, for all point $p \in M$.
Such a $\shP$ is called {\em smooth on $\check{M}$} if 
 $\shP|_{\check{M}}$ is a smooth sub-bundle of the tangent bundle $T\check{M} \otimes \CC$.
\end{definition}

\begin{remark}
In this paper, we only consider such distributions with mild singularities in the sense that they are only singular outside an open dense subset $\check{M}\subset M$. Under our setting, we define smooth sections of singular distributions and  {\em involutive distributions} as follows.
\end{remark}

\begin{definition}\label{def4-4}
Let $\shP$ be a singular complex distribution of $TM\otimes \CC$. For any open subset $U$ of $M$, {\em the space of smooth sections of $\shP$ on $U$} is defined by the smooth section of $TM\otimes \CC$ with value in $\shP$, that is,
$$ \Gamma(U, \shP) = \{ v \in \Gamma(U, TM \otimes \CC) \mid  v_{p} \in (\shP)_{p}, \forall  p\in U \}.$$
\end{definition}

\begin{definition}\label{def4-5}
Let $\shP$ be a singular complex distribution on $M$. $\shP$ is {\em involutive} if it satisfies:
$$[u,v] \in \Gamma(M,\shP), \mathrm{~for~ any~} u,v \in \Gamma(M, \shP).$$
\end{definition}

\begin{definition}\label{def4-7} Let $\shP$ be a singular complex distribution $\shP$ on $M$ and smooth on $\check{M}$. Such a $\shP$ is called a {\em singular polarization on $M$}, if it satisfies the following conditions:
\begin{enumerate} [label = (\alph*)]
\item $\shP$ is involutive, i.e. if $u,v \in \Gamma(M,\shP)$, then $ [u,v] \in \Gamma(M,\shP)$;
\item for every $x \in \check{M}$, $\shP_{p} \subseteq T_{p}M \otimes \CC$ is Lagrangian; and
\item the real rank $ \mathrm{rk}_{\RR}(\shP):=\rank(\shP \cap \overline{\shP} \cap TM)|_{\check{M}}$ is a constant.
\end{enumerate}
 Furthermore, such a singular $\shP$ is called 
 \begin{enumerate}
 \item   [$\cdot$] {\em real polarization}, if $\shP|_{\check{M}} = \overline{\shP}|_{\check{M}}$, i.e. $\mathrm{rk}_{\RR}(\shP|_{\check{M}})=m$;
 \item [$\cdot$]  {\em K\"ahler polarization}, if $\mathcal{P}_{\check{M}}\cap \overline{
\mathcal{P}}|_{\check{M}}=0$ on $\check{M}$, i.e. $r\left( \shP|_{\check{M}}\right) =0$; 
 \item  [$\cdot$]  {\em mixed polarization}, if $0 < \rank(\shP \cap \overline{\shP} \cap TM)|_{\check{M}} < m$, i.e. $0< \mathrm{rk}_{\RR}(\shP|_{\check{M}})<m$.
 \end{enumerate}
\end{definition}
   
    For any point $p \in M$, consider the map $\rho_{p}: T^{n} \rightarrow M$ defined by $\rho_{p}(g)= \rho(g)(p)$. Let $\shI_{\RR} \subset TM$ be the singular distribution generated by fundamental vector fields in $\Im d\rho$, that is $(\shI_{\RR})_{p} = \Im d\rho_{p}(e)$. Let $\shD_{\RR} = (\ker d\mu) \subset TM$ be  a distribution defined by the kernel of $d\mu$. 

\begin{definition}\label{def3-0-1}  
 Let $\shD_{\CC} = \shD_{\RR} \otimes \CC$ and $\shI_{\CC} = \shI_{\RR} \otimes \CC$ be the complexification of $\shD_{\RR}$ and $\shI_{\RR}$ respectively.
 We define the singular distribution $\shP_{\mathrm{mix}} \subset TM \otimes \CC$ by:
\begin{equation}
\shP_{\mathrm{mix}} = (\shP_{J} \cap\shD_{\CC} ) \oplus \shI_{\CC}.
\end{equation}
\end{definition}

\begin{lemma} \label{lem3-0-1}
 $\shI_{\CC}$ and $\shD_{\CC}$ are involutive and smooth on $\check{M}$.
\end{lemma}

\begin{proof}
Since $\shI_{\RR}$ is the singular distribution given by $T^{n}$-orbits, $\shI_{\RR}$ is involutive on $M$ and smooth on $\check{M}$. Therefore the same is true for its complexification $\shI_{\CC}$.
 It's easy to see that $\shD_{\RR}=\ker d\mu$ is smooth on $\check{M}$.  For any $u_{1}, u_{2} \in \Gamma(M, \shD_{\RR})$ and any smooth function $f \in C^{\infty}(\fot^{*}, \RR)$, we have $u_{j} (f\circ \mu)=d\mu(u_{j}) f = 0, j=1,2.$
 It turns out that
$$d\mu([u_{1},u_{2}]) f = [u_{1},u_{2}] (f\circ \mu) = u_{1}u_{2}(f\circ \mu) - u_{2}u_{1}(f\circ \mu) = 0.$$
 This implies $\shD_{\RR}$, therefore $\shD_{\CC}$  is involutive.\end{proof}
  
 \begin{theorem}\label{thm3-0-5}
Under the assumption $(*_{0})$, we have $\shP_{\mathrm{mix}}$ is a singular polarization and $\mathrm{rk}_{\RR}(\shP|_{\check{M}})=n$.
\end{theorem}
 
\begin{proof}
$\shP_{J} \cap \shD_{\CC}$ and $\shI_{\CC}$ are singular distributions on $M$, so $\shP_{\mathrm{mix}} = (\shP_{J} \cap \shD_{\CC}) \oplus \shI_{\CC}$ is a singular distribution on $M$.  
We first show that $\shP_{\mathrm{mix}}$ is involutive case by case. 
\begin{enumerate}
\item If both $u$ and $v \in \Gamma(M, \shP_{J} \cap \shD_{\CC})$, then $[u,v] \in \Gamma(M,\shP_{J} \cap \shD_{\CC})$ follows from the fact that both $\shP_{J}$ and $\shD_{\CC}$ are involutive by Lemma \ref{lem3-0-1} and the integrability of $J$.
\item If both $u$ and $v \in \Gamma(M, \shI_{\CC})$, then $[u,v] \in  \Gamma(M,\shI_{\CC})$ by Lemma \ref{lem3-0-1}.
\item If $u \in \Gamma(M, \shP_{J} \cap \shD_{\CC})$ and $v \in \Gamma(M, \shI_{\CC})$, then $[u,v] \in \shP_{\mathrm{mix}}$. The reason is as follows.
\begin{enumerate}[label = (\roman*)]
\item $[u,v] \in  \Gamma(X,\shP_{J})$ because $\shL_{v}J=0$.
\item $[u,v] \in \Gamma(X, \shD_{\CC})$ since $\shI_{\CC} \subset \shD_{\CC}$ and $\shD_{\CC}^{k}$ is involutive by Lemma \ref{lem3-0-1}.
\item $\shP_{J} \cap \shD_{\CC} \subset \shP_{\mathrm{mix}}$ implies $\Gamma(M, \shP_{J} \cap \shD_{\CC}) \subset \Gamma(M,\shP_{\mathrm{mix}})$.
\end{enumerate}
\end{enumerate} 
By Lemma \ref{lem3-0-1} and the integrability of $J$, $\shP_{\mathrm{mix}}$ is smooth on $\check{M}$. 
Then we prove that $(\shP_{\mathrm{mix}})_{p}$ is a Lagrangian subspace of $T_{p}M \otimes \CC$ for any $p \in \check{M}$. Note that $\shP_{J}$ is holomorphic Lagrangian, one has $\omega(u, v) = 0$, for any $u,v \in \shP_{J}$.
 Since the $T^{n}$-action preserves $\omega$, we have $\omega(u,v) =0$, for $u, v \in \shI_{\CC}$. Recall that $(\shI_{\RR})_{p}$ is a isotropic subspace of $T_{p}M$ orthorgnal to $(\shD_{\RR})_{p}$ with respect to $\omega$, for any point $p \in M$. It follows that $\omega(u,v) = 0$, for $u \in \shP_{J} \cap \shD_{\CC}$, $v \in \shI_{\CC}$. Hence we conclude that $\omega|_{\shP_{\mathrm{mix}}}=0$. It remains to show $\dim (\shP_{\mathrm{mix}})_{p}=m$, for any $p \in \check{M}$. Let $F_{p}=(\shP_{J} \cap \shD_{\CC})_{p}$. By \cite[corollary 2.4]{GS1}, $(F_{p})^{\bot}=(\shP_{J})_{p}+ \shI_{\CC}$ and the sum is direct due to $\shP_{J} \cap \shI_{\CC}=0$. This gives rise to $\dim(F_{p})^{\bot}=(\dim M)/2+n$. We therefore have:
 $$\dim (\shP_{J} \cap \shD_{\CC})_{p}= \dim F_{p}= \dim M-((\dim M)/2+n)= m-n.$$ Adding to the condition $\dim (H_{p})=0$, for any $p \in \check{M}$. We obtain $\dim(\shI_{\CC})_{p}=n$ and $\dim (\shP_{\mathrm{mix}})_{p}=m$.
   Finally, note that $T^{n}$-action on $\check{M}$ is (locally) free, one therefore has $\dim(\shP_{\mathrm{mix}} \cap \bar{\shP}_{\mathrm{mix}}  \cap TM)_{p}= \dim (\Im d\rho)=n$, for any $p \in \check{M}$.
\end{proof}

\begin{remark}
Even though $\shP_{\mathrm{mix}}$ is a singular polarization, $\dim(\shP_{\mathrm{mix}})_{p}$ is always $m$, for any $p \in M$. $\shP_{mix}$ is a singular real polarization, when $n=m$; $\shP_{mix}$ is a singular mixed polarization, when $0< n < m$.
\end{remark}

 \subsection{Construction of a family of complex structures on local models}
 
 In this subsection, we first use the Lie series introduced by Gr\"obner \cite{Gr}, which was used by Mour\~{a}o and Nunes in \cite{MN}, to construct a family of complex structures $J_{t}$ on local models $(T^{n-k}_{\CC} \times \CC^{r})/F$, where $F$ is a finite subgroup of $T^{n}$. Such local models were constructed in Sjamaar's theorem on the existence of holomorphic slices (see Theorem \ref{thm2-2}). According to \cite[Theorem 1.12]{Sj}, for any $p \in M$, we can build a $T^{n}$-equivariant biholomorphic map from a $T^{n}$-invariant neighbourhood $U_{p} \subset T^{n}_{\CC} \times_{H_{p}^{\CC}}\CC^{r}$ around $e=[(1,0)] \in T^{n}_{\CC} \times_{H_{p}^{\CC}}\CC^{r}$ to a $T^{n}$-invariant neighbourhood of $p$ as follows. Since $T^{n}$ is abelian, the $T^{n}$-orbit through $p$ is isotropic and $\mu(p)$ is fixed under the co-adjoint action of $T^{n}$. After shifting the moment map we assume that $\mu(p)=0$. Let $H_{p}$ be the stabilizer of $p$ with respect to the $T^{n}$-action. Then by \cite[Proposition 1.6]{Sj} the stabilizer with respect to the $T^{n}_{\CC}$-action is the complexification $H_{p}^{\CC}$ of $H_{p}$, which has the form of $H_{p}^{\CC}= T_{\CC}^{k} \times F$ with $F$ being a finite subgroup of $T^{n}_{\CC}$. We identify the tangent space $T_{p}M$ at $p$ with $\CC^{m}$. The tangent action of $H_{p}^{\CC}$ defines a linear representation $H_{p}^{\CC} \rightarrow GL(m,\CC)$, the restriction of which to $H_{p}$ is a unitary representation $H_{p} \rightarrow U(m)$. Note that the tangent space to the complex orbit $T^{n}_{\CC}p$ at point $p$ is a complex subspace of $T_{p}M \cong \CC^{m}$. Denote its orthogonal complement by $V$, then $V$ is an $H_{p}^{\CC}$-invariant subspace, which can be identified with $\CC^{r}$ for $r=m-n+k$. According to the proof of \cite[Theorem 1.12]{Sj},
there exists a $T^{n}$-invariant open neighbourhood $$U_{p}=T^{n} exp(\sqrt{-1}D)B \subset T^{n}_{\CC} \times_{H_{p}^{\CC}}\CC^{r}$$ around $e$ and $T^{n}$-equivariant biholomorphic map:
$\phi_{1}: U_{p} \rightarrow M$ such that $\phi_{1}$ is a biholomorphic map onto an open neighbourhood of $T^{n}p$ in $M$ and $\phi_{1}(e)=p$, where $D$ and $B$ are small balls centered at the origins in $T_{p}(T^{n}p)$ and $\CC^{r}$ respectively. 

 Fixing a splitting $T^{n}_{\CC}=T^{n-k}_{\CC} \times T_{\CC}^{k}$, we have a $T^{n}_{\CC}$-equivariant biholomorphism
$$ \phi_{2}: (T_{\CC}^{n-k}\times \CC^{r})/F \rightarrow T^{n}_{\CC}\times_{H_{p}^{\CC}} \CC^{r}, (t,v)\mapsto[(t,1),v],$$
where $T^{n-k}_{\CC}$ acts on $T^{n-k}_{\CC}$ by right multiplication, $T_{\CC}^{k}$ acts on $\CC^{r}$ defining a representation $T_{\CC}^{k} \rightarrow \GL(r,\CC)$ determined by $H_{p} \rightarrow \GL(r,\CC)$, and $F$ is a finite subgroup of $T^{n}_{\CC}$. 
Then $\phi_{1} \circ \phi_{2}: (T_{\CC}^{n-k}\times \CC^{r})/F \rightarrow T^{n}_{\CC}\times_{H_{p}^{\CC}} \CC^{r} \rightarrow M$ induce an $T^{n}$-equivariant isomorphism around a small neighbourhood of $p$ in $M$ with $\phi_{1}\phi_{2}( (1,0))=\phi_{1}(e) =p$.
 
  We first deal with the most essential case, $T^{n-k}_{\CC} \times \CC^{r}$, namely $F$ is trivial.

In our local model $T_{\CC}^{n-k}\times\CC^{r}$ with a Hamiltonian $T^{n}$-action by holomorphic isometries, its complex structure $J$ is standard but its compatible K\"ahler form $\omega$ is not necessary standard. The action by $T^{n}=T^{n-k} \times T^{k}$ is given by the standard multiplication of $T^{n-k}$ on $T_{\CC}^{n-k}$ and a unitary representation of $T^{k}$ on $\CC^{r}$.
  
  More precisely, fix a basis $\{\xi_{1}, \cdots, \xi_{n}\}$ of $\fot_{\ZZ}$ and dual basis $\{\xi_{1}^{*}, \cdots, \xi_{n}^{*}\}$ of $\fot^{*}_{\ZZ}$. 
 Let $$\mu=(\mu_{1},\cdots, \mu_{n}):T^{n-k}_{\CC} \times \CC^{r} \rightarrow \fot^{*} $$
 be the moment map. Let $(w_{1},\cdots,w_{n-k}, z_{1}, \cdots,z_{r})$ be the  standard holomorphic coordinates of $T^{n-k}_{\CC} \times \CC^{r}$. As $T^{n}$ is Abelian, the action of $T^{k}$ on $\CC^{r}$ is diagonalizable:  
  $e^{i\xi_{j}}(z_{1}, \cdots,z_{r})=(e^{i\xi_{j} b_{j1}}z_{1},\cdots, e^{i\xi_{j}b_{jr}}z_{r})$. Let $\varphi: \fot^{*} \rightarrow \RR$ be a strictly convex function and denote the  Hamiltonian vector field associated to $\varphi \circ \mu$ by $X_{\varphi}$. 
  
   We use the Lie series to construct a one-parameter family of complex structures $J_{t}$ around some neighbourhood $V$ of $p=(1,\cdots, 1, 0,\cdots,0)$ as follows. In Theorem \ref{thm3-1-0}, we first confirm that $e^{-itX_{\varphi}}$ can be applied to $w_{j}$ and $z_{l}$ completely (see Definition \ref{def3-2-1}) on $T^{n-k}_{\CC} \times \CC^{r}$, for $ j=1,\cdots,n-k, l=1, \cdots, r$. Moreover we give an explicit formula for $w_{j}^{t}:=e^{-itX_{\varphi}}w_{j}$ and $z^{t}_{l}:=e^{-itX_{\varphi}}z_{l}$. Then in Theorem \ref{thm3-1}, we show that $(dw^{t}\wedge d\bar{w}^{t} \wedge dz^{t}\wedge d\bar{z}^{t})$ is no-where vanishing in a neighbourhood of $p$,$\text{with}~ dw^{t}=\prod_{i=1}^{k}dw^{t}_{i}, dz^{t}=\prod_{i=1}^{r}dz^{t}_{i}$. That is $(TV) \otimes \CC=\shP_{t}\oplus \bar{\shP}_{t}$, where $(\shP_{t})_{p}$ is the linear subspace of $(T_{p}V)\otimes \CC$ spanned by $\frac{\partial}{\partial w_{i}^{t}}$'s and $\frac{\partial}{\partial z_{j}^{t}}$'s, which gives complex structures $J_{t}$ on $V$ with holomorphic coordinates $\{w_{1}^{t}, \cdots,  w_{n-k}^{t}, z_{1}^{t}, \cdots, z_{r}^{t}\}$.

We first introduce the following definition.
\begin{definition}\label{def3-2-1} Let $X$ be a smooth real vector field on $M$. Given a smooth function $f \in C^{\infty}(M)$ and a complex number $\tau \in \mathbb{C}$, we say that $e^{\tau X}$ {\em can be applied to $f$ completely}, if the Lie series $e^{\tau X} f := \sum_{k=0}^{\infty} \frac{(\tau)^{k}}{k!}X^{k}(f) $
is absolutely and uniformly convergent on compact subsets in $M \times \CC$.
\end{definition}
    
 \begin{theorem}\label{thm3-1-0}
Given $T^{n}$ acts on $T^{n-k}_{\CC} \times \CC^{r}$ as described above. Then
$e^{-itX_{\varphi}}$ can be applied to $w_{j}$ and $z_{l}$ completely, for $ j=1,\cdots,n-k, l=1, \cdots, r$. Moreover,
\begin{enumerate}
\item $w_{j}^{t}:=e^{-itX_{\varphi}}w_{j}=w_{j}e^{t\frac{\partial \varphi}{\partial \mu_{j}}}$;
\item $ z_{l}^{t}:=e^{-itX_{\varphi}}z_{l}=z_{l}e^{t(\sum_{j=n-k+1}^{n}\frac{\partial \varphi}{\partial \mu_{j}} b_{jl})}$.
\end{enumerate}
 \end{theorem}
 
  \begin{proof}
  Note that $ \omega(-, \xi_{j}^{\#}) = d\mu_{j}$, where $\xi_{j}^{\#}$ is the fundamental vector field associated to $\xi_{j} \in \fot$. By direct calculations,
 $$\omega \left (-, X_{\varphi} \right )=d(\varphi \circ \mu)=\sum_{j=1}^{n} \frac{\partial \varphi}{\partial \mu_{j}} d\mu_{j}= \sum_{j=1}^{n} \frac{\partial \varphi}{\partial \mu_{j}} \omega(-,  \xi_{j}^{\#})=\omega(-, \sum_{j=1}^{n} \frac{\partial \varphi}{\partial \mu_{j}} \xi_{j}^{\#}).$$

That is $X_{\varphi}=\sum_{j=1}^{n} \frac{\partial \varphi}{\partial \mu_{j}}\xi_{j}^{\#}$.
 \begin{enumerate}
 \item Since $w_{j}$ is nowhere vanishing, $log|w_{j}|$ is well defined. Set $w_{j}=e^{log|w_{j}|+i\theta_{j}}$.
 According to what we assume, $\xi_{j}^{\#}(\theta_{j})=1$, and $\xi_{\gamma}^{\#}(\theta_{j})=0$ for $\gamma=k+1, \cdots,n$. This gives:
 $$-iX_{\varphi} (w_{j})=-iX_{\varphi} (e^{\log|w_{j}|+i\theta_{j}}) =e^{\log|w_{j}|+i\theta_{j}} \frac{\partial \varphi}{\partial \mu_{j}}=w_{j} \frac{\partial \varphi}{\partial \mu_{j}}.$$
 Since $\varphi \circ \mu$ is a $T^{n}$-invariant function, it can be seen that $X_{\varphi} (\frac{\partial \varphi}{\partial \mu_{j}})$=0, for all $j$ and 
$$(-i)^{k}X_{\varphi}^{k}(w_{j})=w_{j}\left(\frac{\partial \varphi}{\partial \mu_{j}}\right)^{k}.$$
 This implies:
\begin{align*} 
e^{-itX_{\varphi}}(w_{j}) &=e^{-itX_{\varphi}}(w_{j})\\
&=\sum_{k} \frac{1}{k!}(-itX_{\varphi})^{k}(w_{j})\\
&=w_{j}\sum_{k} \frac{1}{k!}\left(t\frac{\partial \varphi}{\partial \mu_{j}}\right)^{k}.
\end{align*}
Therefore $e^{-itX_{\varphi}}$ can be applied to $w_{j}$ completely and 
 $$w_{j}^{t}:=e^{-itX_{\varphi}}w_{j}=w_{j}e^{t\frac{\partial \varphi}{\partial \mu_{j}}}.$$
\item Fix $l=1,\cdots,r$, on $z_{l} \ne 0$, let $z_{l} = e^{log|z_{l}|+i \vartheta_{l}}$, for . For our action of $T^{n-k}$ acts trivially on $\CC^{r}$, $\xi_{j}^{\#} \vartheta_{l}=0$ for $j=1,\cdots, n-k$ and $\xi^{\#}_{\gamma}=\sum_{l=1}^{r} b_{\gamma l}\frac{\partial}{\partial \vartheta_{l}}$ for $\gamma=n-k+1, \cdots, n$ with $b_{\gamma l}\in \ZZ$. This implies $$X_{\varphi}(\vartheta_{l})= \sum_{\gamma=n-k+1}^{n}\frac{\partial \varphi}{\partial \mu_{\gamma}} b_{\gamma l},
$$ and 
\begin{align*}
&-iX_{\varphi} z_{l}=-iX_{\varphi} e^{\log|z_{l}|+i\vartheta_{l}}=z_{l} \sum_{\gamma=n-k+1}^{n}\frac{\partial \varphi}{\partial \mu_{\gamma}} b_{\gamma l}.
\end{align*} 
Recall that $\varphi \circ \mu$ and $\mu$ are $T^{n}$-invariant functions. It follows $$X_{\varphi} \left(\sum_{\gamma=n-k+1}^{n}\frac{\partial \varphi}{\partial \mu_{\gamma}} b_{\gamma l}\right)=0,$$ and
$$(-i)^{q}X_{\varphi}^{q}(z_{l})=z_{l}\left(\sum_{\gamma=n-k+1}^{n}\frac{\partial \varphi}{\partial \mu_{\gamma}} b_{\gamma l}\right)^{q}.$$
By the same argument as in $(1)$, we obtain:
\begin{align*}
e^{-itX_{\varphi}} (z_{l})&=e^{-itX_{\varphi}} e^{log|z_{l}|+i \vartheta_{l}}\\
&=z_{l}\sum_{q} \frac{1}{q!}\left(\sum_{\gamma=n-k+1}^{n}t\frac{\partial \varphi}{\partial \mu_{\gamma}} b_{\gamma l} \right)^{q}.
\end{align*}

It can be checked directly that this equality continues to hold true when $z_{l}=0$.
This implies
$e^{-itX_{\varphi}} (z_{l})= \sum_{q}  \frac{1}{q!}(itX_{\varphi})^{q}z_{l}$
is absolutely and uniformly convergent on compact subsets $T^{n-k}_{\CC}\times \CC^{r} \times \RR$. Therefore $e^{-itX_{\varphi}}$ can be applied to $z_{l}$ completely for $l=1,\cdots,r$, and
$$z_{l}^{t}:=e^{-itX_{\varphi}} (z_{l})=z_{l}\sum_{q} \frac{1}{q!}\left(\sum_{\gamma=n-k+1}^{n} t\frac{\partial \varphi}{\partial \mu_{\gamma}} b_{\gamma l}\right)^{q}=z_{l}e^{t(\sum_{\gamma=n-k+1}^{n}\frac{\partial \varphi}{\partial \mu_{\gamma}} b_{\gamma l})}.$$
\end{enumerate}
 \end{proof}
  
  \begin{remark}
   In the analytic setting, Mour\~{a}o and Nunes in \cite[Theorem 2.5]{MN} showed the short time existence of $J_{t}$ for small $t \in \CC$.  \end{remark}
  
 \begin{theorem}\label{thm3-1}
Let $(w_{1},\cdots,w_{n-k}, z_{1}, \cdots,z_{r})$ be the $J$-holomorphic coordinates of $T^{n-k}_{\CC} \times \CC^{r}$. Then for any $t \ge0$, the functions $w_{j}^{t}$'s , $z_{l}^{t}$'s defined in Theorem \ref{thm3-1-0} form a system of complex coordinates on some open neighbourhood $V^{t}$ around $p$, defining a new complex structure $J_{t}$ for which the coordinates $\{w_{1}^{t}, \cdots, w_{n-k}^{t}, z_{1}^{t}, \cdots, z_{r}^{t}\}$ are holomorphic.
\end{theorem}

\begin{proof}
  According to Theorem \ref{thm3-1-0}, the functions $w^{t}_{j}=w_{j}e^{t\frac{\partial \varphi}{\partial \mu_{j}}}$ and $z_{l}^{t}=z_{l}e^{t(\sum_{\gamma=n-k+1}^{n}\frac{\partial \varphi}{\partial \mu_{\gamma}} b_{\gamma l})}$ are smooth. In order to show that $\{w_{1}^{t}, \cdots, w_{n-k}^{t}, z_{1}^{t}, \cdots, z_{r}^{t}\}$ forms a system of complex coordinates on some open set $V$ around $p$, it is enough to show $dw_{1}^{t}, \cdots, dw_{n-k}^{t}$,
  $dz_{1}^{t}, \cdots, dz_{r}^{t}$, $d\bar{w}_{1}^{t}, \cdots, d\bar{w}_{n-k}^{t}, d\bar{z}_{1}^{t}, \cdots, d\bar{z}_{r}^{t}$ are linear independent at $p$. That is $(dw^{t}\wedge d\bar{w}^{t} \wedge dz^{t}\wedge d\bar{z}^{t})(p) \ne 0~~\text{with}~ dw^{t}=\prod_{i=1}^{k}dw^{t}_{i}, dz^{t}=\prod_{i=1}^{r}dz^{t}_{i}$, for any $t \ge0$. This will follow from the convexity of $\varphi$ and the plurisubharmonicity of K\"ahler potential. Set $y_{j}^{t}= log|w_{j}^{t}|$ and $\tilde{w}_{j}^{t}= y^{t}_{j} +i\theta_{j}$, $j=1,\cdots, n-k$. Then we have $w^{t}_{j}=e^{\tilde{w}_{j}^{t}}$, $dw_{j}^{t}=w^{t}_{j}d\tilde{w}_{j}^{t}$, and
  \begin{enumerate}
 \item[$\cdot$]$d\tilde{w}_{j}^{t}= d\tilde{w}_{j}+t d \frac{\partial \varphi}{ \partial \mu_{j}} = d\tilde{w}_{j} +t \left(\frac{\partial^{2} \varphi}{ \partial \mu_{j} \partial\mu_{1}}d\mu_{1} +\cdots +\frac{\partial^{2} \varphi}{ \partial \mu_{j} \partial\mu_{n}}d\mu_{n}\right),j=1,\cdots, n-k$;
\item[$\cdot$] $dz_{l}^{t}
=e^{t(\sum_{\gamma=n-k+1}^{n}\frac{\partial \varphi}{\partial \mu_{\gamma}} b_{\gamma l})}dz_{l}+z_{l}de^{t(\sum_{\gamma=n-k+1}^{n}\frac{\partial \varphi}{\partial \mu_{\gamma}} b_{\gamma l})}$, 
  $l=1,\cdots r$.
\end{enumerate}

 Since $T^{k}$ acts on $\CC^{r}$ defining a unitary representation, we have $dz_{l}^{t}
=e^{t(\sum_{\gamma=n-k+1}^{n}\frac{\partial \varphi}{\partial \mu_{\gamma}} b_{\gamma l})}dz_{l}$ at point $p=(1,0)$. This gives:
\begin{align*}
&\begin{pmatrix}
d\tilde{w}_{1}^{t}\\
\vdots \\
d\tilde{w}_{n-k}^{t}\\
d\bar{\tilde{w}}_{1}^{t}\\
\vdots \\
d\bar{\tilde{w}}_{n-k}^{t}\\
dz_{1}^{t}\\
d\bar{z}_{1}^{t}\\
\vdots \\
dz_{r}^{t}\\
d\bar{z}_{r}^{t}
\end{pmatrix}(p)
=\begin{pmatrix}
I+tH_{\varphi} A&tH_{\varphi} A& tH_{\varphi} B \\
t H_{\varphi} A&I+tH_{\varphi} A&tH_{\varphi} B\\
0&0&K
\end{pmatrix}
\begin{pmatrix}
d\tilde{w}_{1}\\
\vdots \\
d\tilde{w}_{n-k}\\
d\bar{\tilde{w}}_{1}\\
\vdots \\
d\bar{\tilde{w}}_{n-k}\\
dz_{1}\\
d\bar{z}_{1}\\
\vdots \\
dz_{r}\\
d\bar{z}_{r}
\end{pmatrix}(p).
\end{align*}
 where  
 $A=\begin{pmatrix}
 \frac{\partial \mu_{1}}{\partial \tilde{w}_{1} }& \cdots&
\frac{\partial \mu_{1}}{\partial \tilde{w}_{n-k} }\\
\cdots&\cdots&\cdots\\
\frac{\partial \mu_{n-k}}{\partial \tilde{w}_{1} }&\cdots&
\frac{\partial \mu_{n-k}}{\partial \tilde{w}_{n-k} }
 \end{pmatrix}$,
 $H_{\varphi}= \begin{pmatrix}
\frac{\partial ^{2} \varphi}{ \partial \mu_{1} \partial \mu_{1}}& \cdots&
 \frac{\partial ^{2} \varphi}{ \partial \mu_{1} \partial \mu_{n-k}}\\
 \cdots& \cdots& \cdots\\
\frac{\partial ^{2} \varphi}{ \partial \mu_{n-k} \partial \mu_{1}}&\cdots&
\frac{\partial ^{2} \varphi}{ \partial \mu_{n-k} \partial \mu_{n-k}}
 \end{pmatrix}$, and\\
$K= \begin{pmatrix}
 e^{t(\sum_{\gamma=n-k+1}^{n}\frac{\partial \varphi}{\partial \mu_{\gamma}} b_{\gamma 1})} & & \\
    & \ddots & \\
    & & e^{t(\sum_{\gamma=n-k+1}^{n}\frac{\partial \varphi}{\partial \mu_{\gamma}} b_{\gamma r})} \end{pmatrix}$ .
 
Since $(w_{1},\cdots,w_{n-k}, z_{1}, \cdots,z_{r})$ is the $J$-holomorphic coordinates of $T^{n-k}_{\CC} \times \CC^{r}$,
  $(dw\wedge d\bar{w} \wedge dz\wedge d\bar{z})(p) \ne 0.$ In order to show $(dw^{t}\wedge d\bar{w}^{t} \wedge dz^{t}\wedge d\bar{z}^{t})(p) \ne 0$, for any $t \ge 0$, it is enough to show that 
 $$\begin{pmatrix}
I+tH_{\varphi} A&tH_{\varphi} A\\
tH_{\varphi} A&I+tH_{\varphi} A
\end{pmatrix}
$$ is non-degenerate at $p$. 
By Burns-Guillemin's results in \cite{BG} or see Theorem \ref{thm2-0}, we can choose a $T^{n-k}$-invariant K\"ahler potential $\rho=\rho(|w_{1}|, \cdots, |w_{n-k}|, z_{1}, \bar{z}_{1},\cdots,z_{m},\bar{z}_{m}),$
such that $\mu_{j}= \frac{\partial \rho}{\partial \tilde{w}_{j}}=w_{j}\frac{\partial \rho}{\partial w_{j}}=\bar{w}_{j}\frac{\partial \rho}{\partial \bar{w}_{j}}= \frac{\partial \rho}{\partial \tilde{\bar{w}}_{j}}$, for $j=1, \cdots, n-k$. One has that $$H_{\rho}= \begin{pmatrix}
  \frac{\partial^{2} \rho}{\partial \bar{w}_{1} \partial w_{1}}& \cdots&
  \frac{\partial^{2} \rho}{\partial \bar{w}_{n-k} \partial w_{1}} & \frac{\partial^{2} \rho}{ \partial \bar{z}_{1} \partial w_{1}}& \cdots&
  \frac{\partial^{2} \rho}{ \partial \bar{z}_{m} \partial w_{1}}\\
  \cdots&  \cdots&  \cdots&  \cdots&  \cdots&  \cdots\\
 \frac{\partial^{2} \rho}{\partial \bar{w}_{1} \partial w_{n-k}}&\cdots&
  \frac{\partial^{2} \rho}{\partial \bar{w}_{n-k} \partial w_{n-k}}&
   \frac{\partial^{2} \rho}{ \partial \bar{z}_{1} \partial w_{n-k}} &\cdots&
  \frac{\partial^{2} \rho}{ \partial \bar{z}_{m} \partial w_{n-k}}\\
   \frac{\partial^{2} \rho}{\partial \bar{w}_{1} \partial z_{1}}&\cdots&
  \frac{\partial^{2} \rho}{\partial \bar{w}_{n-k} \partial z_{1}}&
  \frac{\partial^{2} \rho}{ \partial \bar{z}_{1} \partial z_{1}}& \cdots&
  \frac{\partial^{2} \rho}{ \partial \bar{z}_{m} \partial z_{1}}\\
    \cdots&  \cdots&  \cdots&  \cdots&  \cdots&  \cdots\\
  \frac{\partial^{2} \rho}{\partial \bar{w}_{1} \partial z_{m}}&\cdots&
  \frac{\partial^{2} \rho}{\partial \bar{w}_{n-k} \partial z_{m}}&
  \frac{\partial^{2} \rho}{ \partial \bar{z}_{1} \partial z_{m}}& \cdots&
  \frac{\partial^{2} \rho}{ \partial \bar{z}_{m} \partial z_{m}}
 \end{pmatrix} $$ 
 is positive definite by the plurisubharmonicity of the K\"ahler potential $\rho$. In particular, the $(n-k)\times (n-k)$ principal minor 
 $$\begin{pmatrix}
\frac{\partial^{2} \rho}{\partial \bar{w}_{1} \partial w_{1}}& \cdots&
  \frac{\partial^{2} \rho}{\partial \bar{w}_{n-k} \partial w_{1}} \\
  \cdots&  \cdots&  \cdots\\
 \frac{\partial^{2} \rho}{\partial \bar{w}_{1} \partial w_{n-k}}&\cdots&
  \frac{\partial^{2} \rho}{\partial \bar{w}_{n-k} \partial w_{n-k}}
 \end{pmatrix}$$
  is positive definite. By straightforward computations,
\begin{align*}
 &A=\begin{pmatrix}
 \frac{\partial \mu_{1}}{\partial \tilde{w}_{1} }& \cdots&
\frac{\partial \mu_{1}}{\partial \tilde{w}_{n-k} }\\
\cdots&\cdots&\cdots\\
\frac{\partial \mu_{n-k}}{\partial \tilde{w}_{1} }&\cdots&
\frac{\partial \mu_{n-k}}{\partial \tilde{w}_{n-k} }
 \end{pmatrix}= \begin{pmatrix}
 \frac{\partial \mu_{1}}{\partial \bar{\tilde{w}}_{1} }& \cdots&
\frac{\partial \mu_{1}}{\partial \bar{\tilde{w}}_{n-k} }\\
\cdots&\cdots&\cdots\\
\frac{\partial \mu_{n-k}}{\partial \bar{\tilde{w}}_{1} }&\cdots&
\frac{\partial \mu_{n-k}}{\partial \bar{\tilde{w}}_{n-k} }
 \end{pmatrix}\\
 &=\frac{1}{4}\begin{pmatrix}
\frac{\partial^{2} \rho}{\partial y_{1} \partial y_{1}}& \cdots&
 \frac{\partial^{2} \rho}{\partial y_{n-k} \partial y_{1}} \\
  \cdots&  \cdots&  \cdots\\
 \frac{\partial^{2} \rho}{\partial y_{1} \partial y_{n-k}}&\cdots&
  \frac{\partial^{2} \rho}{\partial y_{n-k} \partial y_{n-k}}
 \end{pmatrix}= \begin{pmatrix}
\frac{\partial^{2} \rho}{\partial \bar{\tilde{w}}_{1} \partial \tilde{w}_{1}}& \cdots&
  \frac{\partial^{2} \rho}{\partial \bar{\tilde{w}}_{n-k} \partial \tilde{w}_{1}} \\
  \cdots&  \cdots&  \cdots\\
 \frac{\partial^{2} \rho}{\partial \bar{\tilde{w}}_{1} \partial \tilde{w}_{n-k}}&\cdots&
  \frac{\partial^{2} \rho}{\partial \bar{\tilde{w}}_{n-k} \partial \tilde{w}_{n-k}}
 \end{pmatrix} \\
 &=  \begin{pmatrix} 
  w_{1} & &\\
&\ddots&\\
&&  w_{n-k}\end{pmatrix}   \begin{pmatrix}
\frac{\partial^{2} \rho}{\partial \bar{w}_{1} \partial w_{1}}& \cdots&
  \frac{\partial^{2} \rho}{\partial \bar{w}_{n-k} \partial w_{1}} \\
  \cdots&  \cdots&  \cdots\\
 \frac{\partial^{2} \rho}{\partial \bar{w}_{1} \partial w_{n-k}}&\cdots&
  \frac{\partial^{2} \rho}{\partial \bar{w}_{n-k} \partial w_{n-k}}
 \end{pmatrix} \begin{pmatrix} 
  \bar{w}_{1} & &\\
&\ddots&\\
&&  \bar{w}_{n-k}\end{pmatrix}.
 \end{align*}
 Since $0\ne (w_{1}, \cdots, w_{n-k}) \in T^{n-k}_{\CC}$, $A$ is positive-definite due to the positive-definiteness of $H_{\rho}$. As $\varphi$ is strictly convex, $H_{\varphi}$ is positive definite. The positive-definiteness of $A$ and $H_{\varphi}$ implies $H_{\varphi}A$ is positive definite.
 It follows that
 $\begin{pmatrix}
I+tH_{\varphi} A&tH_{\varphi} A\\
tH_{\varphi} A&I+tH_{\varphi} A
\end{pmatrix}
$
is invertible, for all $t \ge0$. Hence we have 
$$(dw^{t}\wedge d\bar{w}^{t} \wedge dz^{t}\wedge d\bar{z}^{t})(p) \ne 0, \text{with}~ dw^{t}=\prod_{i=1}^{k}dw^{t}_{i}, dz=\prod_{i=1}^{r}dz^{t}_{i}.$$
This implies that there exists some open neighbourhood $V$ around $p$ such that $dw^{t}\wedge d\bar{w}^{t} \wedge dz^{t}\wedge d\bar{z}^{t}$
 is nowhere vanishing on $V^{t}$. Therefore, the functions
$w_{j}^{t}$'s, $z_{l}^{t}$'s,
form a system of complex coordinates on $V^{t}$, defining a new complex structure $J_{t}$ on $V^{t}$ for which the coordinates $w_{j}^{t}$'s, $z_{l}^{t}$'s are holomorphic.
\end{proof}

In the above local model $T_{\CC}^{n-k}\times\CC^{r}$ with a Hamiltonian $T^{n}$-action by holomorphic isometries, its complex structure J is standard but its compatible K\"ahler form $\omega$ is not necessary standard. The action by $T^{n}=T^{n-k} \times T^{k}$ is given by the standard multiplication of $T^{n-k}$ on $T_{\CC}^{n-k}$ and a unitary representation of $T^{k}$ on $\CC^{r}$. 
Let $F$ be a subgroup of $T^{n}$ acting on $T_{\CC}^{n-k}\times\CC^{r}$ freely. Assume the $T^{n}$-action commuting with $F$ can descend to a $T^{n}$-action on $(T_{\CC}^{n-k}\times\CC^{r})/F$.
 Let $\varphi: \fot^{*} \rightarrow \RR$ be a strictly convex function and denote the  Hamiltonian vector field associated to $\varphi \circ \mu$ by $X_{\varphi}$.
 We are going to construct a family of complex structures $J_{t}$ by applying $e^{-itX_{\varphi}}$ to the coordinate functions in the following theorem.

\begin{theorem}\label{thm3-1-1}
Suppose that $T^{n}$ acts on $ ((T^{n-k}_{\CC} \times \CC^{r})/F,\omega, J)$ as described above. Then there is a one-parameter family of complex structures $J_{t}$ on the neighbourhood of $p=[(1,0)] \in (T^{n-k}_{\CC} \times \CC^{r})/F$ by applying $e^{-itX_{\varphi}}$ to  the $J$-holomorphic coordinates.
\end{theorem}
\begin{proof}
Denote $\tilde{U}= T^{n-k}_{\CC} \times \CC^{r}$. By the assumption, we have a $T^{n}$-equiviarant unramified finite cover $\pi: T^{n-k}_{\CC} \times \CC^{r} \rightarrow (T^{n-k}_{\CC} \times \CC^{r})/F$.
Let $X_{\tilde{\varphi}}$ be the Hamiltonian vector field associated to $\pi^{*}(\varphi \circ \mu)$ on $T^{n-k}_{\CC} \times \CC^{r}$ with respect to $\pi^{*}\omega$. 
 By Theorem \ref{thm3-1}, we can construct a one family of complex structures $J_{t}$ by applying $e^{-itX_{\tilde{\varphi}}}$ to  the $J$-holomorphic coordinates, on the neighbourhood of $\pi^{-1}(p)$. 
 These can be descended to $(T^{n-k}_{\CC} \times \CC^{r})/F$, as $F$ acts on $T^{n-k}_{\CC} \times \CC^{r}$ freely and commute with the $T^{n}$-action.
\end{proof}

\subsection{Commuting formula}
In the last subsection, we obtain a one-parameter family of complex structures on local models. In order to glue these complex structures on local models, 
we also need to prove the following commuting formula
$$ \label{eq-com}e^{itg\frac{\partial}{\partial \theta_{j}}}f(z)=f(e^{itg\frac{\partial}{\partial \theta_{j}}}z).$$
(see Theorem \ref{thm3-2-0}). This type of formulae was first studied by Gr\"obner in \cite{Gr} for the case of holomorphic differential operators and later used by Mour\~{a}o and Nunes in \cite{MN} for this kind of gluing problems for small time $t$.
Let's start with the following basic lemmas.
\begin{lemma}\label{lem3-1-1}
 If $e^{itX}$ can be applied to $f, g \in C^{\infty}(M)$ completely, then  $e^{itX}$ can be applied to $f+g$ and $fg$ completely. Moreover
\begin{enumerate}
\item[(i)] $e^{itX}(f +g)=e^{itX}f +e^{itX}g$;
\item[(ii)]$e^{itX}(fg)=(e^{itX}f)(e^{itX}g)$.
\end{enumerate}
\end{lemma}

\begin{proof}
\begin{enumerate}
\item[(i)] It is obvious that $X^{v}(f+g)=X^{v}f+X^{v}g$ for any $v \in \NN$.  For any $0 \ll N \in \NN$,
\begin{align*}
\left|\sum_{v=0}^{N} \frac{(it)^{v}}{v!} X^{v}(f+g)- e^{itX}f -e^{itX}g\right| \le \left|\sum_{v=0}^{N} \frac{(it)^{v}}{v!}X^{v}(f)-e^{itX}f \right|+\left|\sum_{v=0}^{N} \frac{(it)^{v}}{v!}X^{v}(g)-e^{itX}g\right|
\end{align*}
By the assumption, $ \sum_{v=0}^{\infty} \frac{(it)^{v}}{v!}X^{v}(f) $ and $ \sum_{v=0}^{\infty} \frac{(it)^{v}}{v!} X^{v}(g)$ are absolutely and uniformly convergent on compact subsets in $M \times \RR$.
 It follows that  $e^{itX}$ can be applied to $f+g$ completely and 
$e^{itX}(f +g)=e^{itX}f +e^{itX}g$.
\item[(ii)] Using $X(fg)=(Xf)g+f(Xg)$, we have
$X^{v}(fg)=\sum_{l=0}^{v}\binom{v}{l}(X^{l}f)(X^{v-l}g)$. Therefore:

\begin{align*}
e^{itX} (fg)= \sum_{v=0}^{\infty} \frac{(it)^{v}}{v!}X^{v}(fg)& = \sum_{v=0}^{\infty}\sum_{l=0}^{v} \frac{(it)^{v}}{l!(v-l)!}(X^{l}f)(X^{v-l}g)\\
&=\sum_{l=0}^{\infty} \frac{(it)^{l}}{l!}X^{l}f\sum_{k=0}^{\infty} \frac{(it)^{k}}{k!}X^{k}g =(e^{itX}f)(e^{itX}g).\end{align*}

Here we have used the absolutely and uniformly convergence of $ \sum_{v=0}^{\infty} \frac{(it)^{v}}{v!}X^{v}(f) $ and $ \sum_{v=0}^{\infty} \frac{(it)^{v}}{v!} X^{v}(g)$ on compact subsets in $M \times \RR$.
 It follows that  $e^{itX}$ can be applied to $fg$ completely and 
 $e^{itX}(fg)=(e^{itX}f)(e^{itX}g)$.
\end{enumerate}
\end{proof}

\begin{lemma}\label{lem3-2-0}
Let $X_{1},X_{2}$ be two commuting vector fields. If $e^{itX_{1}}$, $e^{itX_{2}}$ and $e^{it(X_{1}+X_{2})}$ can be applied to $f \in C^{\infty}(M)$ completely, then $e^{itX_{1}}$ can be applied to $e^{itX_{2}}f$ completely and $ e^{itX_{1}}(e^{itX_{2}}f)=e^{it(X_{1}+X_{2})}f$.
\end{lemma}

\begin{proof}
Since $X_{1}$ and $X_{2}$ commute,
$(X_{1}+X_{2})^{v}(f)=\sum_{l=0}^{v}\binom{v}{l}(X_{1}^{l}f)(X_{2}^{v-l}g)$, for $v \in \NN$.

 Observe that, for any $0 \ll N \in \NN$,
\begin{align*}
& \left|\sum_{v=0}^{N}\sum_{l=0}^{v} \frac{(it)^{v}}{l!(v-l)!}X_{1}^{l}(X_{2}^{v-l}f)- \sum_{l=0}^{N} \frac{(it)^{l}}{l!}X_{1}^{l}(\sum_{k=0}^{N} \frac{(it)^{k}}{k!}X_{2}^{k}f)\right| \\
\le &  \sum_{l=N/2}^{\infty} \sum_ {k=0}^{N/2} \frac{(it)^{l+k}}{l!(k)!}\left|X^{l}(X^{k}f)\right| + \sum_{k=N/2}^{\infty}\sum_ {l=0}^{N/2}\frac{(it)^{l+k}}{l!(k)!}\left|X^{l}(X^{k}f)\right|. 
\end{align*}

Using the absolutely and uniformly convergence of $ \sum_{v=0}^{\infty} \frac{(it)^{v}}{v!}X^{v}(f) $ and $ \sum_{v=0}^{\infty} \frac{(it)^{v}}{v!} X^{v}(g)$ on compact subsets in $M \times \RR$, one has

\begin{align*} e^{it(X_{1}+X_{2})} (f)&= \sum_{v=0}^{\infty} \frac{(it)^{v}}{v!}X^{v}(fg)=  \sum_{v=0}^{\infty}\sum_{l=0}^{v} \frac{(it)^{v}}{l!(v-l)!}X_{1}^{l}(X_{2}^{v-l}f)\\
&=\sum_{l=0}^{\infty} \frac{(it)^{l}}{l!}X_{1}^{l}\left(\sum_{k=0}^{\infty} \frac{(it)^{k}}{k!}X_{2}^{k}f\right) =e^{itX_{1}}(e^{itX_{2}}f).\end{align*}
 Therefore $e^{itX_{1}}$ can be applied to $e^{itX_{2}}f$ completely and $e^{itX_{1}}(e^{itX_{2}}f)=e^{it(X_{1}+X_{2})}f$.
\end{proof}
 
\begin{lemma}\label{lem3-1-2} 
Let $f$ be a local holomorphic function on $\CC^{m}$. Assume that
 $e^{it\frac{\partial}{\partial \theta_{j}}}$ can be applied to $f$ and coordinate functions $z_{1},\cdots, z_{m}$ completely with $z_{j}=r_{j}e^{i\theta_{j}}$,  for any $t\in \RR$. Suppose that $f$ still converges at the point $(e^{it\frac{\partial}{\partial \theta_{j}}}z_{1},\cdots, e^{it\frac{\partial}{\partial \theta_{j}}}z_{m})=: e^{it\frac{\partial}{\partial \theta_{j}}}z$. Then we have:
$$\left(e^{it\frac{\partial}{\partial \theta_{j}}}f\right)\left(z\right)=f\left(e^{it\frac{\partial}{\partial \theta_{j}}}z\right).$$
\end{lemma}

\begin{proof} 
Denote $D_{j}=iz_{j} \frac{\partial}{\partial z_{j}}$. Since $f$ is a holomorphic function and $\frac{\partial}{\partial \theta_{j}}=D_{j}+\bar{D}_{j}$, it can be seen that
$D_{j}^{v}f=(\frac{\partial}{\partial \theta_{j}})^{v}(f)$, for any $v \in \NN$. This implies: $ \sum_{v}\frac{(it)^{v}}{v!}(\frac{\partial}{\partial \theta_{j}})^{v}(f)= \sum_{v}\frac{(it)^{v}}{v!}D_{j}^{v}f$. We are then able to conclude:
$$e^{it\frac{\partial}{\partial \theta_{j}}}f(z)=e^{itD_{j}}f(z), ~~f\left(e^{it\frac{\partial}{\partial \theta_{j}}}z\right)=f\left(e^{itD_{j}}z\right).$$
 By Theorem 6 of \cite{Gr}, $e^{itD_{j}}f(z)=f(e^{itD_{j}}z).$ It follows
$$\left(e^{it\frac{\partial}{\partial \theta_{j}}}f\right)\left(z\right)=f\left(e^{it\frac{\partial}{\partial \theta_{j}}}z\right).$$
\end{proof}

\begin{theorem}\label{thm3-2-0}
Let $f$ be a local holomorphic function on $\CC^{m}$ in $z$ with $z_{j}=r_{j}e^{i\theta_{j}}$.  Assume that $e^{it\frac{\partial}{\partial \theta_{j}}}$ and $e^{itg\frac{\partial}{\partial \theta_{j}}}$ can be applied to $f$ and coordinate functions $z_{1},\cdots, z_{n}$ on $\CC^{m}$ completely. Let $g$ be a real smooth function such that $\frac{\partial}{\partial \theta_{j}}g=0$. Suppose that $f$ converges at the point $\left(e^{it\frac{\partial}{\partial \theta_{j}}}z_{1},\cdots, e^{it\frac{\partial}{\partial \theta_{j}}}z_{m}\right)=: \left(e^{it\frac{\partial}{\partial \theta_{j}}}z\right)$ and $\left(e^{itg\frac{\partial}{\partial \theta_{j}}}z_{1},\cdots, e^{itg\frac{\partial}{\partial \theta_{j}}}z_{m}\right)=: \left(e^{itg\frac{\partial}{\partial \theta_{j}}}z\right)$. Then we have:
\begin{equation} \label{eq-com}\left(e^{it\frac{\partial}{\partial \theta_{j}}}f\right)\left(z\right)=f\left(e^{it\frac{\partial}{\partial \theta_{j}}}z\right).\end{equation}\end{theorem}

\begin{proof} As before
$D_{j}^{v}f=(\frac{\partial}{\partial \theta_{j}})^{v}f$. Since $\frac{\partial}{\partial \theta_{j}}g=0$, we have $(g\frac{\partial}{\partial \theta_{j}})^{v}f=g^{v}(\frac{\partial}{\partial \theta_{j}})^{v}f$, for any $v \in \NN$. This gives rise to:
$$ e^{itg\frac{\partial}{\partial \theta_{j}}}f= \sum_{v}\frac{(itg)^{v}}{v!}(\frac{\partial}{\partial \theta_{j}})^{v}f= \sum_{v}\frac{(itg)^{v}}{v!}D_{j}^{v}f.$$ 

As $f(z)$ is holomorphic in $z$, expressed as
$f(z_{1},\cdots, z_{m})=\sum_{I} b_{I}z^{I}, \text{with}~I=(i_{1}, \cdots,i_{m})$, the truncated polynomial  
$$f_{k}(z)= \sum_{|I| \le k} b_{I}z^{I},$$
always converge to $f(z)$ as $k \rightarrow \infty$. We also have $\lim_{k\rightarrow \infty} \frac{\partial}{\partial z_{j}}f_{k} (z)=\frac{\partial f(z)}{\partial z_{j}}$ and similar ones for higher derivatives hold. 
 As $f(z)$ is well-defined at $(e^{it\frac{\partial}{\partial \theta_{j}}}z)$ and $(e^{itg\frac{\partial}{\partial \theta_{j}}}z)$, we have:

$$\lim_{k\rightarrow \infty}f_{k}(z)=f(z) ,\lim_{k\rightarrow \infty}f_{k}(e^{it\frac{\partial}{\partial \theta_{j}}}z)=f(e^{it\frac{\partial}{\partial \theta_{j}}}z),\lim_{k\rightarrow \infty}f_{k}(e^{itg\frac{\partial}{\partial \theta_{j}}}z)=f(e^{itg\frac{\partial}{\partial \theta_{j}}}z).$$

As $D_{j}=iz_{j} \frac{\partial}{\partial z_{j}}$ is a holomorphic differential operator, 
$$\lim_{k\rightarrow \infty}D_{j}^{v}f_{k}(z)=D_{j}^{v}f(z), 
e^{itg\frac{\partial}{\partial \theta_{j}}}f=\sum_{v}\frac{(itg)^{v}}{v!}D_{j}^{v}f =  \sum_{v}\frac{(itg)^{v}}{v!}\lim_{k\rightarrow \infty}D_{j}^{v}f_{k}.$$
On the other hand, by Lemma \ref{lem3-1-1}, one obtains:
$f_{k}(e^{it\frac{\partial}{\partial \theta_{j}}}z)=e^{it\frac{\partial}{\partial \theta_{j}}} f_{k}(z)= \sum_{v}\frac{(it)^{v}}{v!}D_{j}^{v}f_{k}(z)$, and $f_{k}(e^{itg\frac{\partial}{\partial \theta_{j}}}z)=e^{itg\frac{\partial}{\partial \theta_{j}}} f_{k}(z)= \sum_{v}\frac{(itg)^{v}}{v!}D_{j}^{v}f_{k}(z)$. It follows that:
$$f(e^{itg\frac{\partial}{\partial \theta_{j}}}z)=\lim_{k\rightarrow \infty}f_{k}(e^{itg\frac{\partial}{\partial \theta_{j}}}z)=\lim_{k\rightarrow \infty}\sum_{v}\frac{(itg)^{v}}{v!}D_{j}^{v}f_{k}(z).$$
By Theorem 6 of \cite{Gr} or Lemma \ref{lem3-1-2}, it can be seen that: for any $t >0$
\begin{equation}\label{eqcov} \sum_{v}\frac{(it)^{v}}{v!}\lim_{k\rightarrow \infty}D_{j}^{v}f_{k}=\lim_{k\rightarrow \infty}\sum_{v}\frac{(it)^{v}}{v!}D_{j}^{v}f_{k}.\end{equation}
Since the smooth function $g$ is a real, we can replace $t$ by $gt$ in equation (\ref{eqcov}) and obtain
$$  \sum_{v}\frac{(itg)^{v}}{v!}\lim_{k\rightarrow \infty}D_{j}^{v}f_{k}=\lim_{k\rightarrow \infty}\sum_{v}\frac{(itg)^{v}}{v!}D_{j}^{v}f_{k}.$$
Therefore,
$$  e^{itg\frac{\partial}{\partial \theta_{j}}}f(z)=  \sum_{v}\frac{(itg)^{v}}{v!}\lim_{k\rightarrow \infty}D_{j}^{v}f_{k}(z)=\lim_{k\rightarrow \infty}f_{k}(e^{itgD_{j}}z) =f(e^{itg\frac{\partial}{\partial \theta_{j}}}z).$$
\end{proof}

\subsection{Construction of $\{J_{t}\}_{t\ge0}$ via gluing local models}
In this subsection, we glue the complex structures on local models to construct a one-parameter family of complex structures $J_{t}$ on $M$ in Theorem \ref{thm3-3} under the assumption $(*)$. Furthermore, we show that $J_{t}$ is compatible with $\omega$ and the corresponding path of K\"ahler metrics $g_{t}=\omega(-,J_{t}-)$ is a complete geodesic ray in the space of K\"ahler metrics of $M$ in Theorem \ref{thm3-4-1}.

\begin{theorem}\label{thm3-3}
Under the assumption $(*)$, for any $t >0$, there exists a complex structure $J_{t}$ given by applying $e^{-itX_{\varphi}}$ to $J$-holomorphic coordinates and a unique biholomorphism:
 $$\psi_{t}: (M,J_{t}) \rightarrow (M,J).$$
\end{theorem}

\begin{theorem}\label{thm3-4-1}
Under the assumption $(*)$, for any $t \ge 0$, $(M,\omega, J_{t})$ is a K\"ahler manifold. Moreover the path of K\"ahler metrics $g_{t}=\omega(-,J_{t}-)$ is a complete geodesic ray in the space of K\"ahler metrics of $M$. 
\end{theorem}

\begin{proof}[Proof of Theorem \ref{thm3-3}]
 According to \cite[Theorem 1.12]{Sj}, for any $p \in M$, we can build a $T^{n}$-equivariant biholomorphic map from a $T^{n}$-invariant neighbourhood $U_{p} \subset T^{n}_{\CC} \times_{H_{p}^{\CC}}\CC^{r}$ around $e=[(1,0)] \in T^{n}_{\CC} \times_{H_{p}^{\CC}}\CC^{r}$ to a $T^{n}$-invariant neighbourhood of $p$ as follows. Since $T^{n}$ is abelian, the $T^{n}$-orbit through $p$ is isotropic and $\mu(p)$ is fixed under the co-adjoint action of $T^{n}$. After shifting the moment map we assume that $\mu(p)=0$. Let $H_{p}$ be the stabilizer of $p$ with respect to the $T^{n}$-action. Then by \cite[Proposition 1.6]{Sj} the stabilizer with respect to the $T^{n}_{\CC}$-action is the complexification $H_{p}^{\CC}$ of $H_{p}$, which has the form of $H_{p}^{\CC}= T_{\CC}^{k} \times F$ with $F$ being a finite subgroup of $T^{n}_{\CC}$. We identify the tangent space $T_{p}M$ at $p$ with $\CC^{m}$. The tangent action of $H_{p}^{\CC}$ defines a linear representation $H_{p}^{\CC} \rightarrow GL(m,\CC)$, the restriction of which to $H_{p}$ is a unitary representation $H_{p} \rightarrow U(m)$. Note that the tangent space to the complex orbit $T^{n}_{\CC}p$ at point $p$ is a complex subspace of $T_{p}M \cong \CC^{m}$. Denote its orthogonal complement by $V$, then $V$ is an $H_{p}^{\CC}$-invariant subspace, which can be identified with $\CC^{r}$ for $r=m-n+k$. According to the proof of \cite[Theorem 1.12]{Sj},
there exists a $T^{n}$-invariant open neighbourhood $$U_{p}=T^{n} exp(\sqrt{-1}D)B \subset T^{n}_{\CC} \times_{H_{p}^{\CC}}\CC^{r}$$ around $e$ and $T^{n}$-equivariant biholomorphic map:
$\phi_{1}: U_{p} \rightarrow M$ such that $\phi_{1}$ is a biholomorphic map onto an open neighbourhood of $T^{n}p$ in $M$ and $\phi_{1}(e)=p$, where $D$ and $B$ are small balls centered at the origins in $T_{p}(T^{n}p)$ and $\CC^{r}$ respectively.

Then we divide the proof into two steps.
 We will first show that for any point $p \in M$, there exists a $T^{n}$-invariant holomorphic chart $V_{p}$ around $p$ and applying $e^{-itX_{\varphi}}$ to holomorphic coordinates gives us a family of complex structures $J_{t}$ in some neighbourhood $V_{p}^{t} $. Next we show that the local complex structures $\{(V_{p}^{t}, J_{t})\}'s$ can be glued together defining complex structures $J_{t}$ on $M$, for all $t >0$. Fixing a splitting $T^{n}_{\CC}=T^{n-k}_{\CC} \times T_{\CC}^{k}$, we have a $T^{n}_{\CC}$-equivariant biholomorphism
$$ \phi_{2}: (T_{\CC}^{n-k}\times \CC^{r})/F \rightarrow T^{n}_{\CC}\times_{H_{p}^{\CC}} \CC^{r}, (t,v)\mapsto[(t,1),v],$$
where $T^{n-k}_{\CC}$ acts on $T^{n-k}_{\CC}$ by right multiplication, $T_{\CC}^{k}$ acts on $\CC^{r}$ defining a representation $T_{\CC}^{k} \rightarrow \GL(r,\CC)$ determined by $H_{p} \rightarrow \GL(r,\CC)$, and $F$ is a finite subgroup of $T^{n}_{\CC}$.

\begin{enumerate}
\item When $F$ is trivial, let $(z_{1} \cdots,z_{m})$ be the standard $J$-holomorphic coordinates on $T^{n-k}_{\CC}\times \CC^{r}$. For any $t >0$, by Theorem \ref{thm3-1-0} and Theorem \ref{thm3-1}, the functions
$$ z_{l}^{t}=e^{-itX_{\varphi}}z_{l},~ l=1, \cdots, m,$$ form a system of complex coordinates on some open neighbourhood $V_{p}^{t}$ of $e$, defining a new complex structure $J_{t}$ on $V_{p}^{t}$ for which the coordinates $\{ z_{1}^{t}, \cdots, z_{m}^{t}\}$ are holomorphic.
\item For general $F$, let $(z_{1} \cdots,z_{m})$ be the standard $J$-holomorphic coordinates on $V_{p}\subset (T_{\CC}^{n-k}\times \CC^{r})/F$. For any $t >0$, by Theorem \ref{thm3-1-1}, the functions
$$ z_{l}^{t}=e^{-itX_{\varphi}}z_{l},~ l=1, \cdots, m,$$ form a system of complex coordinates on open neighbourhoods $V_{p}^{t}$ of $e$, defining a new complex structure $J_{t}$ on $V_{p}^{t}$ for which the coordinates $\{ z_{1}^{t}, \cdots, z_{m}^{t}\}$ are holomorphic.

\end{enumerate}
The last step is to glue these complex structures on local models to define complex structures $J_{t}$ on $M$.
For $t>0$, let $\{V_{p}, z_{p}\}_{p \in M}$ be $J$-holomorphic local charts constructed as above.  Let $\phi_{\alpha \beta}$'s be the coordinate transition functions, that is, $\phi_{\alpha \beta}$'s are biholomorphic functions and 
$$z_{p_{\alpha}} = \phi_{\alpha \beta} \circ z_{p_{\beta}}, \text{with}~  z_{p_{\beta}}:V_{p_{\beta}} \rightarrow \foV_{p_{\beta}} \subset \CC^{m} $$
Similar to Theorem \ref{thm3-1-0}, we can show that $e^{it\xi_{j}^{\#}}$
and $e^{it \frac{\partial \varphi}{\partial \mu_{j}}\xi_{j}^{\#}}$ can be applied to $\{z_{p}\}$ completely. Note that $\xi_{j}^{\#} (\frac{\partial \varphi}{\partial \mu_{j}})=0$. By Theorem \ref{thm3-2-0}, we have 
$$e^{it \frac{\partial \varphi}{\partial \mu_{j}}\xi_{j}^{\#}} z_{p_{\alpha}} = \phi_{\alpha \beta} \circ ( e^{it \frac{\partial \varphi}{\partial \mu_{j}}\xi_{j}^{\#}} z_{p_{\beta}}).$$
Since $\varphi \circ \mu$ and $\mu$ are $T^{n}$-invariant functions, $ \frac{\partial \varphi}{\partial \mu_{1}}\xi_{1}^{\#}, \cdots, \frac{\partial \varphi}{\partial \mu_{n}}\xi_{n}^{\#}$ commute with each other. Then by Lemma \ref{lem3-2-0}, we obtain:
$$z_{p_{\alpha}}^{t}=e^{-itX_{\varphi}}z_{p_{\alpha}}= e^{-itX_{\varphi}} ( \phi_{\alpha \beta} \circ z_{p_{\beta}})=\phi_{\alpha \beta} \circ (e^{-itX_{\varphi}} z_{p_{\beta}})=\phi_{\alpha \beta} \circ z_{p_{\beta}}^{t}.$$
Note that $\phi_{\alpha \beta}$ independent of $t$.
This implies that $\{V_{p_{\alpha}}^{t}, z_{p_{\alpha}}^{t}\}$'s is a new atlas on $M$ with the local transition functions $\phi_{\alpha \beta}$'s on $V_{p_{\alpha}}^{t} \cap V_{p_{\beta}}^{t}$. Therefore $\{V_{p_{\alpha}}^{t}, z_{p_{\alpha}}^{t}\}$'s define a new complex structure and we have the following commuting diagram:
$$
\begin{tikzcd}
&     \foV_{p_{\alpha}}^{t} \arrow{dd}{\phi_{\beta \alpha}} \arrow{rr}{id} &&\foV_{p_{\alpha}} \arrow{dd}{\phi_{\beta \alpha}} &\\
V_{p_{\alpha}}^{t}\cap V_{p_{\beta}}^{t} \arrow{ur}{z^{t}_{p_{\alpha}}} \arrow{dr}{z_{p_{\beta}}^{t}} &&&&V_{p_{\alpha}}\cap V_{p_{\beta}} \arrow{ul}{z_{p_{\alpha}}} \arrow{dl}{z_{p_{\beta}}}\\
&   \foV_{p_{\beta}}^{t} \arrow{rr}{id} && \foV_{p_{\beta}} &
\end{tikzcd}.
$$
We define $\psi_{t,\beta}= z_{p_{\beta}}^{-1} \circ z_{p_{\beta}}^{t}: V_{p_{\beta}}^{t} \rightarrow V_{p_{\beta}}$. The above commuting diagram guarantees that $\{\psi_{t,\beta}\}$
can be glued together to obtain a well defined global biholomorphism $$\psi_{t}: (M, J_{t}) \rightarrow (M,J).$$
It is easy to see the inverse map $\psi_{-t}$ exists on the chart $\psi_{t}(V_{p_{\beta}}^{t})$, we have:
$$\psi_{-t,\beta}= (z_{p_{\beta}}^{t})^{-1} \circ z_{p_{\beta}}, \text{and}~~e^{itX_{\varphi}}z^{t}_{p_{\beta}}=z_{p_{\beta}}.$$ 
Therefore, $\psi_{t}$ is the unique biholomorphism from $(M,J_{t})$ to $(M,J_{0})$ such that:
$z_{p_{\beta}}^{t}= z_{p_{\beta}}\circ \psi_{t,\beta}$.
\end{proof}

Denote $J_{0}=J$, according to Theorem \ref{thm3-3}, $\psi_{t}: (M,J_{t}) \rightarrow (M,J_{0})$ is a biholomorphism, on local $J_{0}$ holomorphic coordinates, acts as $e^{-it X_{\varphi}}$. Observe that $J_{t}=\psi_{t}^{*}J_{0}$ with $\psi_{0}=id$. Since $(M,J_{0}, \omega)$ is a K\"ahler manifold and $X_{\varphi}$ is a Hamiltonian vector field. Then we confirm that $J_{t} $ is compatible with $\omega$, and the path of K\"ahler metrics $g_{t}=\omega(-,J_{t}-)$ is a complete geodesic ray in the space of K\"ahler metrics of $M$ as follows.

\begin{proof}[Proof of Theorem \ref{thm3-4-1}]
To prove that $(M,\omega, J_{t})$ is a K\"ahler manifold, it is enough to show that $\omega$ is of type $(1,1)$ with respect to $J_{t}$ and the Riemannian metric $g_{t}(-,-)= \omega(-, J_{t}(-))$ is positive definite, for all $t >0$.
Let $z_{j}^{0}$'s be the holomorphic coordinates on an open set $U$ of $M$ with respect to $J_{0}=J$ and $z_{j}^{t}= e^{-itX_{\varphi}}z_{j}^{0}$ be the holomorphic coordinates with respect to $J_{t}$. Since $(M,\omega, J_{0})$ is a K\"ahler manifold, $\omega$ is of type $(1,1)$ with respect to $J_{0}$, which is equivalent to $$\omega(X_{z_{i}^{0}},X_{z_{j}^{0}})=0, \text{for all}~~ i,j,$$
where $X_{z_{j}^{0}}$ is the Hamiltonian vector field associated to $z_{j}^{0}$. Since $\shL_{X_{\varphi}} \omega=0$, following the argument in \cite[Theorem 4.1]{MN}, one has $\{ z^{t}_{i},z^{t}_{j}\}=0$ for $i,j=1,\cdots,m$.
 That is $$dz^{t}_{i}(X_{z^{t}_{j}})= \omega(X_{X_{z^{t}_{j}}}, X_{X_{z^{t}_{i}}}) =0, ~\text{for all} ~~i,j.$$
 Therefore $\omega$ is of type $(1,1)$ with respect to $J_{t}$ and $g_{t}=\omega(-,J_{t}-)$ is a pseudo-K\"ahler metric. By the assumption, $g_{0}=\omega(-,J_{0}-)$ is a K\"ahler metric. This implies $g_{t}$ is also K\"ahler metric, for all $t >0$.
 
 It remains to show $g_{t}$ is a geodesic path in the space of K\"ahler metrics of $M$. In \cite{MN}, it was proved that $g_{t}$, given by imaginary time flow, satisfies the geodesic equation in the space of K\"ahler metrics, which is equivalent to the one studied by Donaldson. As our imaginary time flow exists for all  $t>0$, $\{g_{t}\}_{t \in \RR_{+}}$ is  a complete geodesic ray.
  \end{proof}
 
 \subsection{A family of K\"ahler polarizations $\shP_{t}$ degenerating to polarization $\shP_{\mathrm{mix}}$}
 In this subsection, we study the relation between K\"ahler 
 polarization $\shP_{J}$ associated to complex structure the $J$ and $\shP_{\mathrm{mix}}$ (see Theorem \ref{thm3-0-5}) constructed in subsection \ref{sec1}. Similar problems have been investigated for other symplectic manifolds including symplectic vector space in \cite{KW}, cotangent bundles of compact Lie groups in \cite{FMMN1, FMMN2, Hal}, toric varieties in \cite{BFMN}, and flag manifold in \cite{HK}. 
 
\begin{theorem}\label{thm3-5}
Under the assumption $(*)$,
let $J_{t}$ be the one-parameter family of complex structures constructed in Theorem \ref{thm3-3}. Then we have
$$\lim_{t\rightarrow \infty} \shP_{t}= \shP_{\mathrm{mix}}.$$
That is, $\lim_{t\rightarrow \infty} (\shP_{t})_{p}= (\shP_{\mathrm{mix}})_{p}$,  where the limit is taken in the Lagrangian Grassmannian of the complexified tangent space at point $p \in M$. 
\end{theorem}

\begin{proof}
As discussed in the proof of \ref{thm3-3}, for any $p \in M$,
there exists a holomorphic chart $V_{p} \subset (T_{\CC}^{n-k}\times \CC^{r})/F$ of $p$ and the standard $J$-holomorphic coordinates $(w_{1}, \cdots, w_{n-k}, z_{1}, \cdots, z_{r})$ on $V_{p}$  such that
\begin{enumerate}
\item[(i)]With respect to the splitting $T^{n}_{\CC}= T^{n-k}_{\CC} \times T^{k}_{\CC}$ discussed in proof of Theorem \ref{thm3-3}, $T_{\CC}^{n-k} \subset T_{\CC}^{n}$ acts on $T^{n-k}_{\CC}$ by right multiplication and $ T_{n}^{k} \subset T_{\CC}^{n}$ acts on $\CC^{r}$ defining a linear unitary representation $T^{k} \rightarrow U(r)$ with $r=m-n+k$;
\item[(ii)]the functions
$w_{j}^{t}= e^{-itX_{\varphi}}w_{j} = e^{y_{j}^{t} + i\theta_{j}},  z_{l}^{t}=e^{-itX_{\varphi}}z_{l} ,  j=1,\cdots,n-k, l=1, \cdots, r,$ are holomorphic coordinates on $V_{p}^{t}$ with respect to $J_{t}$. 
\end{enumerate}

Let $\xi_{1},\cdots,\xi_{n}$ be a basis of $\fot$ such that $\xi_{1}, \cdots,\xi_{n-k}$ and $\xi_{n-k+1}, \cdots,\xi_{n}$ span the Lie algebras of $T^{n-k}$ and $ T^{k}$ respectively. We denote the fundamental vector field associated to $\xi_{j}$ by $\xi^{\#}_{j}$. One has:
$d\mu_{j}= \omega(-,\xi^{\#}_{j})$, for $j=1,\cdots,n$ and $(d\mu_{j})_{p}=0$, for $j=n-k+1, \cdots, n$. 
It turns out that
\begin{enumerate}
\item [$\cdot$]$(\ker d\mu)_{p} \otimes \CC=  \left\{\frac{\partial }{\partial \theta_{1}} , \cdots, \frac{\partial }{\partial \theta_{n-k}}\right\}_{\CC} \oplus \left\{ \frac{\partial }{\partial z_{1}},\cdots, \frac{\partial }{\partial z_{r}},\frac{\partial }{\partial \bar{z}_{1}},\cdots, \frac{\partial }{\partial \bar{z}_{r}} \right\}_{\CC}$;
\item [$\cdot$] $(\ker d\mu)_{p} \otimes \CC \cap \shP_{J}= (\ker d\mu )_{p} \otimes \CC \cap TM^{0,1}_{J}= \left\{ \frac{\partial }{\partial \bar{z}_{1}},\cdots, \frac{\partial }{\partial \bar{z}_{r}} \right\}_{\CC}$;
\item [$\cdot$] $\mathrm({\Im  d\rho})_{p}\otimes \CC =\left\{\xi^{\#}_{1} , \cdots,\xi_{n-k}^{\#}\right\}_{\CC} =\left\{\frac{\partial }{\partial \theta_{1}} , \cdots, \frac{\partial }{\partial \theta_{n-k}}\right\}_{\CC}$.\\
Therefore, by the construction of $\shP_{\mathrm{mix}}$, one has:
$$(\shP_{\mathrm{mix}})_{p}=(\mathrm{Ker} d\mu \otimes\CC \cap \shP_{J}) \oplus(\Im d\rho \otimes \CC) = \left\{\frac{\partial }{\partial \theta_{1}} , \cdots, \frac{\partial }{\partial \theta_{n-k}}\right\}_{\CC} \oplus \left\{ \frac{\partial }{\partial \bar{z}_{1}},\cdots, \frac{\partial }{\partial \bar{z}_{r}} \right\}_{\CC}.$$
\end{enumerate}
Then we will focus on computing $(\shP_{t})_{p}$.
By theorem \ref{thm3-1}, we have:
\begin{enumerate}
\item $w_{j}^{t}=w_{j}e^{t\frac{\partial \varphi}{\partial \mu_{j}}}=e^{y_{j}^{t} +i\theta_{j}}, j=1,\cdots,n-k;$
\item $z_{l}^{t}= z_{l}e^{t(\sum_{\gamma=n-k+1}^{n}\frac{\partial \varphi}{\partial \mu_{\gamma}} b_{\gamma l})}, l=1,\cdots,r$, with $0\le r\le m$.
  \end{enumerate}
Since $\frac{\partial }{\partial \bar{z}_{l}^{t}}=\frac{\partial }{\partial \bar{z}_{l}}$ at $p=[(1,0)]$, for $l=1,\cdots, r$, one has:
  \begin{align*}
 & (\shP_{t})_{p}=\left\{ \frac{\partial }{\partial \bar{w}_{1}^{t}}, \cdots, \frac{\partial }{\partial \bar{w}_{n-k}^{t}},\frac{\partial }{\partial \bar{z}_{1}^{t}},\cdots, \frac{\partial }{\partial \bar{z}_{r}^{t}}\right\}_{\CC}\\
  &= \left\{ \frac{\partial }{\partial \bar{w}_{1}^{t}}, \cdots, \frac{\partial }{\partial \bar{w}_{n-k}^{t}}\right\}_{\CC} \oplus \left\{\frac{\partial }{\partial \bar{z}_{1}^{t}},\cdots, \frac{\partial }{\partial \bar{z}_{r}^{t}}\right\}_{\CC}\\
    &= \left\{ \frac{\partial }{\partial y_{1}^{t}}+i\frac{\partial }{\partial \theta_{1}}, \cdots,\frac{\partial }{\partial y_{n-k}^{t}}+i\frac{\partial }{\partial \theta_{n-k}}\right\}_{\CC} \oplus \left\{\frac{\partial }{\partial \bar{z}_{1}},\cdots, \frac{\partial }{\partial \bar{z}_{r}}\right\}_{\CC}.
 \end{align*}
 Therefore it is enough to show that $ \left\{ \frac{\partial }{\partial y_{j}^{t}}+i\frac{\partial }{\partial \theta_{j}}\right\}_{\CC} \rightarrow  \left\{\frac{\partial }{\partial \theta_{j}}\right\}_{\CC}$, as $t \rightarrow \infty$ for $j=1,\cdots,n-k$.\\
By Burns-Guillemin's theorem in \cite{BG} (or see Theorem \ref{thm2-0}), we can choose a $T^{n-k}$-invariant K\"ahler potential $$\rho=\rho(|w_{1}|, \cdots, |w_{n-k}|, z_{1}, \bar{z}_{1},\cdots,z_{m},\bar{z}_{m}),$$ 
such that $\mu_{j}= \frac{\partial \rho}{\partial \tilde{w}_{j}}=w_{j}\frac{\partial \rho}{\partial w_{j}}=\bar{w}_{j}\frac{\partial \rho}{\partial \bar{w}_{j}}= \frac{\partial \rho}{\partial \tilde{\bar{w}}_{j}}$, for $j=1, \cdots, n-k$. 
 By the plurisubharmonicity of the K\"ahler potential $\rho$, $$\begin{pmatrix}
\frac{\partial^{2} \rho}{\partial \bar{w}_{1} \partial w_{1}}& \cdots&
  \frac{\partial^{2} \rho}{\partial \bar{w}_{n-k} \partial w_{1}} \\
  \cdots&  \cdots&  \cdots\\
 \frac{\partial^{2} \rho}{\partial \bar{w}_{1} \partial w_{n-k}}&\cdots&
  \frac{\partial^{2} \rho}{\partial \bar{w}_{n-k} \partial w_{n-k}}
 \end{pmatrix}$$
  is positive definite as discussed in the proof of Theorem \ref{thm3-1}. Note that
\begin{align*}
 &\frac{1}{2} \begin{pmatrix}
\frac{\partial \mu_{1}}{\partial y_{1}}& \cdots&
\frac{\partial \mu_{n-k}}{\partial y_{1}}\\
\cdots & \cdots &\cdots\\
\frac{\partial \mu_{1}}{\partial y_{n-k}}& \cdots&
\frac{\partial \mu_{n-k}}{\partial y_{n-k}}
\end{pmatrix}
 &=  \begin{pmatrix} 
  w_{1} & &0\\
&\ddots&\\
0&&  w_{n-k}\end{pmatrix}   \begin{pmatrix}
\frac{\partial^{2} \rho}{\partial \bar{w}_{1} \partial w_{1}}& \cdots&
  \frac{\partial^{2} \rho}{\partial \bar{w}_{n-k} \partial w_{1}} \\
  \cdots&  \cdots&  \cdots\\
 \frac{\partial^{2} \rho}{\partial \bar{w}_{1} \partial w_{n-k}}&\cdots&
  \frac{\partial^{2} \rho}{\partial \bar{w}_{n-k} \partial w_{n-k}}
 \end{pmatrix} \begin{pmatrix} 
  \bar{w}_{1} & &0\\
&\ddots&\\
0&&  \bar{w}_{n-k}\end{pmatrix}.
 \end{align*}
Since $w_{1}(p)\ne 0, \cdots, w_{n-k}(p)\ne 0$, one obtains:
$ \begin{pmatrix}
\frac{\partial \mu_{1}}{\partial y_{1}}& \cdots&
\frac{\partial \mu_{n-k}}{\partial y_{1}}\\
\cdots & \cdots &\cdots\\
\frac{\partial \mu_{1}}{\partial y_{n-k}}& \cdots&
\frac{\partial \mu_{n-k}}{\partial y_{n-k}}
\end{pmatrix}
$ is positive definite. It turns out that $(\mu_{1},\cdots,\mu_{n-k}, \theta_{1},\cdots,\theta_{n-k},z_{1},\cdots,z_{r},\bar{z}_{1},\cdots,\bar{z}_{r})$ are the local coordinates on the open set $V_{p}$ containing $p$. And $\begin{pmatrix}
\frac{\partial y_{1}}{\partial \mu_{1}}&\cdots&
\frac{\partial y_{n-k}}{\partial \mu_{1}}\\
\cdots&\cdots&\cdots\\
\frac{\partial y_{1} }{\partial \mu_{n-k}}& \cdots&
\frac{\partial y_{n-k}}{\partial \mu_{n-k}}
\end{pmatrix}
= \begin{pmatrix}
\frac{\partial \mu_{1}}{\partial y_{1}}& \cdots&
\frac{\partial \mu_{n-k}}{\partial y_{1}}\\
\cdots & \cdots &\cdots\\
\frac{\partial \mu_{1}}{\partial y_{n-k}}& \cdots&
\frac{\partial \mu_{n-k}}{\partial y_{n-k}}
\end{pmatrix}^{-1}$ is positive definite.  For any $t >0$,

 \begin{align*}
A_{t}^{-1}: =
\begin{pmatrix}
\frac{\partial y_{1}^{t}}{\partial \mu_{1}}&\cdots&
\frac{\partial y_{n-k}^{t}}{\partial \mu_{1}}\\
\cdots&\cdots&\cdots\\
\frac{\partial y_{1}^{t} }{\partial \mu_{n-k}}& \cdots&
\frac{\partial y_{n-k}^{t}}{\partial \mu_{n-k}}
\end{pmatrix}
=\begin{pmatrix}
\frac{\partial y_{1}}{\partial \mu_{1}}&\cdots&
\frac{\partial y_{n-k}}{\partial \mu_{1}}\\
\cdots&\cdots&\cdots\\
\frac{\partial y_{1} }{\partial \mu_{n-k}}& \cdots&
\frac{\partial y_{n-k}}{\partial \mu_{n-k}}
\end{pmatrix}
+ t\begin{pmatrix}
 \frac{\partial^{2} \varphi}{ \partial \mu_{1} \partial\mu_{1}} & \cdots&
 \frac{\partial^{2} \varphi}{ \partial \mu_{1} \partial\mu_{n-k}}\\  
 \cdots & \cdots&\cdots\\
 \frac{\partial^{2} \varphi}{ \partial \mu_{n-k} \partial\mu_{1}} & \cdots & \frac{\partial^{2} \varphi}{ \partial \mu_{n-k} \partial\mu_{n-k}}
 \end{pmatrix}.
\end{align*}
 As $\varphi$ is strictly convex function, 

 $$\lim_{t \rightarrow \infty} A_{t}= \lim_{t \rightarrow \infty}\begin{pmatrix}
\frac{\partial y_{1}^{t}}{\partial \mu_{1}}&\cdots&
\frac{\partial y_{n-k}^{t}}{\partial \mu_{1}}\\
\cdots&\cdots&\cdots\\
\frac{\partial y_{1}^{t} }{\partial \mu_{n-k}}& \cdots&
\frac{\partial y_{n-k}^{t}}{\partial \mu_{n-k}}
\end{pmatrix}^{-1}
= \lim_{t \rightarrow \infty}\begin{pmatrix}
\frac{\partial \mu_{1}}{\partial y^{t}_{1}}& \cdots&
\frac{\partial \mu_{n-k}}{\partial y^{t}_{1}}\\
\cdots & \cdots &\cdots\\
\frac{\partial \mu_{1}}{\partial y^{t}_{n-k}}& \cdots&
\frac{\partial \mu_{n-k}}{\partial y^{t}_{n-k}}
\end{pmatrix}=0.$$
That is, $$  \lim_{t\rightarrow \infty}\frac{\partial \mu_{i}}{\partial y^{t}_{j}} = 0,~\text{for}~ i,j=1,\cdots,n-k.$$
Then one has:
 $$\frac{\partial}{\partial y^{t}_{i}} = \sum_{j=1}^{n-k} \frac{\partial \mu_{j}}{\partial y^{t}_{i}} \frac{\partial}{\partial \mu_{j}} \rightarrow 0,  \text{as} ~~t\rightarrow \infty.$$
 This gives: 
$$ \left\{ \frac{\partial }{\partial y_{j}^{t}}+i\frac{\partial }{\partial \theta_{j}}\right\}_{\CC} \rightarrow  \left\{\frac{\partial }{\partial \theta_{j}}\right\}_{\CC}, ~~j=1,\cdots,n-k,$$
where the limit is taken inside the Grassmannian of the linear subspaces inside  $T_{p}M\otimes \CC$.
 We therefore have:
 \begin{align*}
 &\lim_{t \rightarrow \infty} (\shP_{t})_{p}=\lim_{t \rightarrow \infty}\left\{ \frac{\partial }{\partial \bar{w}_{1}^{t}}, \cdots, \frac{\partial }{\partial \bar{w}_{n-k}^{t}},\frac{\partial }{\partial \bar{z}_{1}^{t}},\cdots, \frac{\partial }{\partial \bar{z}_{r}^{t}}\right\}_{\CC}\\
    &= \lim_{t \rightarrow \infty}\left\{ \frac{\partial }{\partial y_{1}^{t}}+i\frac{\partial }{\partial \theta_{1}}, \cdots,\frac{\partial }{\partial y_{n-k}^{t}}+i\frac{\partial }{\partial \theta_{n-k}}\right\}_{\CC} \oplus \lim_{t \rightarrow \infty}\left\{\frac{\partial }{\partial \bar{z}_{1}},\cdots, \frac{\partial }{\partial \bar{z}_{r}}\right\}_{\CC} \\
    &= \left\{\frac{\partial }{\partial \theta_{1}} , \cdots, \frac{\partial }{\partial \theta_{n-k}}\right\}_{\CC} \oplus \left\{ \frac{\partial }{\partial \bar{z}_{1}},\cdots, \frac{\partial }{\partial \bar{z}_{r}} \right\}_{\CC},
 \end{align*}
where the limit is taken in the Lagrangian Grassmannian of the complexified tangent space at point $p \in M$.
\end{proof}  
\phantomsection

\bibliographystyle{amsplain}

\end{document}